\documentclass{article}
\usepackage[T1]{fontenc}
\usepackage{lmodern}  
\usepackage{cite}
\usepackage{amsmath}
\usepackage{appendix}
\usepackage{amsfonts}
\usepackage{amssymb}
\usepackage{graphicx}
\usepackage{xcolor}
\usepackage{mathabx}
\usepackage{geometry}
\usepackage{enumitem}
\usepackage{hyperref}
\usepackage{comment}
\usepackage{subcaption}
\usepackage{amsthm}
\newcommand{\Hi}{\RR^N}

\newcommand{\Id}{\mathrm{Id}}
\newcommand{\RR}{\mathbb{R}}
\newcommand{\prox}{\mathrm{prox}}
\newcommand{\g}{\mathrm{g}}
\newcommand{\dom}{\mathrm{dom}~}
\newcommand{\diam}{\mathrm{diam}}
\newcommand{\Fix}{\mathrm{Fix}~}

\newcommand{\vertiii}[1]{{\left\vert\kern-0.25ex\left\vert\kern-0.25ex\left\vert #1 
    \right\vert\kern-0.25ex\right\vert\kern-0.25ex\right\vert}}

\DeclareMathOperator*{\argmin}{arg\,min}

\newtheorem{lemma}{Lemma}
\newtheorem{definition}{Definition}
\newtheorem{proposition}{Proposition}
\newtheorem{theorem}{Theorem}
\newtheorem{remark}{Remark}
\newtheorem{assumption}{Assumption}

\title{Characterizations of inexact proximal operators}
\author{Guillaume Lauga, Samuel Vaiter}
\date{\today}

\begin{document}

\maketitle
\begin{abstract}
Proximal operators are now ubiquitous in non-smooth optimization. Since their introduction in the seminal work of Moreau, many papers have shown their effectiveness on a wide variety of problems, culminating in their use to construct convergent deep learning methods. The characterization of these operators for non-convex penalties was completed recently in \cite{gribonval2020characterization}. In this paper, we propose to follow this line of work by characterizing inexact proximal operators, thus providing an answer to what constitutes a good approximation of these operators.
We propose several definitions of approximations and discuss their regularity, approximation power, and their fixed points.
Equipped with these characterizations, we investigate the convergence of proximal algorithms in the presence of errors that may be non-summable and/or non-vanishing. In particular, we look at the proximal point algorithm, and at the forward-backward, Peaceman-Rachford and Douglas-Rachford algorithms when we minimize the sum of a weakly convex function (whose proximal operator is approximated) and a strongly convex function.

\end{abstract}
\section{Introduction}
Proximal operators are a central tool in modern non-smooth optimization. Since their introduction in the seminal work of Moreau \cite{moreau1962CR,moreau1963CR,moreau1965proximite}, many works have shown their effectiveness on a wide variety of problems \cite{bauschke2017,parikh2014proximal,gribonval2020characterization}, culminating today in their use to construct convergent deep learning methods \cite{hurault2022proximal,hurault2023relaxed,ryu2019plug,Reehorst2019REDClarifications}. The characterization of these operators for non-convex penalties was completed in \cite{gribonval2020characterization}. In this paper, we propose to follow these steps by characterizing inexact proximal operators, thus providing an answer to what constitutes a good approximation of these operators.

\paragraph{Context of the paper.} Our study sits in the context of optimization problems formulated as the minimization of the sum of two functions,
\begin{equation} \label{eq:optim}
    \widehat{x} \in \argmin_{x\in\Hi} f(x) + \phi(x).
\end{equation}
In typical imaging or machine learning applications, $f$ is a smooth data-fidelity term and $\phi$ is a typically non-smooth regularization term promoting certain properties on the solution (e.g., sparsity in a given basis \cite{chambolle2016introduction}). These regularizations are dealt with by operator splitting, computing optimization steps on $f$, and optimization steps on $\phi$. Optimization steps on $\phi$ are usually done with its proximal operator of parameter $\lambda>0$. A proximal step consists in finding for $y\in\Hi$,
\begin{equation} \label{eq:prox_operator}
    \prox_{\lambda \phi}(y) \in \argmin_{x \in \Hi} \Phi_\lambda(x) = \left\{ \phi(x) + \frac{1}{2\lambda} \Vert x - y \Vert^2 \right\}.
\end{equation}
Many splitting methods rely on proximal operators, since their first introduction in optimization by Martinet \cite{martinet1970regularisation,martinet1972determination} and Rockafellar \cite{rockafellar1976augmented,rockafellar1976monotone}. To cite the most common: proximal gradient descent or forward-backward (FB) splitting \cite{combettes2005signal,lions1979splitting,beck2009fast}, Douglas-Rachford (DR) splitting \cite{douglas1956numerical,lions1979splitting,eckstein1992douglas,bauschke2017}, alternating direction method of multipliers (ADMM) \cite{gabay1983applications,fortin1983decomposition,combettes2012primal,parikh2014proximal}, primal-dual (PD) algorithms \cite{chambolle2011first,vu2013splitting,condat2013primal}, and many others \cite{eckstein1989splitting,bauschke2017}. 

The computation of proximal operators involves a minimization problem which, for a variety of reasons, can introduce inexactness in the algorithms. %
If an explicit formulation of the proximal operator is available in a lot of cases \cite{parikh2014proximal,bauschke2017,chierchia_proximity}, most penalties are not "proximable", i.e., there does not exist an explicit formulation for their proximal operator (e.g., Total Variation (TV) based penalties \cite{beck2009fastTV,lauga2024iml}). Therefore, the computation is replaced by an estimation procedure. This procedure introduces an error that needs to be managed along the iterations in order to retain convergence of the proximal algorithm to a solution of \eqref{eq:optim} \cite{salzo2012inexact,villa2013accelerated,barre2023principled,schmidt2011convergence,gu2018inexact}. Such inaccuracies have been covered for the aforementioned algorithms for instance in \cite{salzo2012inexact,villa2013accelerated,barre2023principled,schmidt2011convergence,gu2018inexact} (FB and its accelerated versions), \cite{alves2020relative,svaiter2019weakly} (DR), \cite{alves2020relative,eckstein2017approximate} (ADMM), \cite{condat2013primal,rasch2020inexact} (PD). 

Other avenues of errors are quantization errors \cite{xu2019signprox}, a phenomenon occurring in distributed optimization \cite{elgabli2020q,zhu2016quantized,chen2012fast} or in online optimization \cite{dixit2019online,bastianello2021distributed,ajalloeian2020inexact}.  
A last avenue of errors can be referred to as \textit{learning errors}. Recent advances in the image restoration community have seen the replacement of an explicit regularizer $\phi$ by learned priors \cite{bredies2024learning,terris2021learning,hurault2022proximal,hurault2023relaxed,bredies2024learning,pesme2025map}, through the replacement of the minimizing operator (e.g., the proximal operator) associated to $\phi$ by off-the-shelf, built or learned, denoisers \cite{venkatakrishnan2013plug,hertrich2025learning}. In this context, learned denoiser aims at approximating some optimization operator related to a true functional that is not accessible. Hence, their use introduces an error in the optimization with respect to this true functional. The learned denoisers have replaced proximal operators in FB algorithms \cite{ryu2019plug,terris2021learning,hurault2022proximal,hurault2023relaxed,hurault2023convergent}, DR splitting \cite{hurault2022proximal,hurault2023relaxed}, ADMM algorithms \cite{venkatakrishnan2013plug,chan2016plug} and PD splittings \cite{ono2017primal}.

A question arises: \emph{can we still guarantee the convergence to a solution of the minimization problem \eqref{eq:optim}, when the proximal operator of $\phi$ is inexact and its error cannot be brought to zero?} 

\paragraph{Related works.} To provide an answer we need to understand how to model the errors. Several models have been proposed in the literature, we present them briefly below.

To model errors the authors of \cite{burachik1997enlargement,solodov2000error} adopt the point of view of monotone inclusions (find $x\in \Hi$ such that $0 \in T(x)$) and enlarge the maximal monotone operator $T$ (a monotone operator whose graph is maximal \cite[Definition 20.20]{bauschke2017}) by relaxing the monotonicity assumption with a parameter $\epsilon\geq 0$. We have
\begin{equation}
    T^\epsilon(x) = \{v \in \Hi \mid \forall y \in \Hi, u \in T(y), \langle v-u, x-y \rangle \geq - \epsilon\}.
\end{equation}
This definition includes the target operator $T= T^0$, and an approximate solution to the inclusion is given by
\begin{equation}
    \begin{cases}
   & v \in T^\epsilon(x) \\
   & \Vert \gamma v + y-x \Vert \text{ and } \epsilon \text{ are small},
    \end{cases}
\end{equation}
where $\gamma>0$ is a regularization parameter. Such modelling has direct connection with later works by taking $T = \partial \phi$ and $T^\epsilon = \partial_\epsilon \phi$ the convex approximate subdifferential \cite{brondsted1965subdifferentiability,hiriart-urruty1993convexII}, which we denote $\epsilon$-subdifferential $\partial_\epsilon \phi$. Indeed, several works have refined the approach of \cite{solodov2000error} for convex optimization (for instance \cite{monteiro2010convergence}).

Building on this enlargement of the subdifferential, two natural approximations of the proximal operators of convex penalties arise: one by finding a $0$ of the $\epsilon$-subdifferential of the proximal optimization problem \eqref{eq:prox_operator} \cite{rockafellar1976monotone,auslender1987numerical,guler1992new,cominetti1997coupling,salzo2012inexact},
\begin{equation}
    \text{Find } z \text{ such that } 0 \in \partial_{\epsilon/\lambda} \Phi_\lambda(z),
\end{equation}
and one by finding a element of the $\epsilon$-subdifferential of $\phi$,
\begin{equation}
    \text{Find } z \text{ such that }\frac{y-z}{\lambda} \in \partial_{\epsilon/\lambda} \phi(z).
\end{equation}
This last approximation allows authors of \cite{salzo2012inexact,villa2013accelerated} to construct a procedure to obtain and qualify an approximation, 
and to build inexact proximal gradient and accelerated inexact proximal gradient algorithms to solve non-smooth convex optimization problems \cite{villa2013accelerated}.  Authors of \cite{bednarczuk2023calculus,bednarczuk2023forward,khanh2025inexact} present similar constructions for weakly convex penalties, by extending the notion of approximate subdifferential to weakly convex functions. 

In \cite{osher2023hamilton}, the authors exploit the fact that the Moreau envelope is a viscosity solution to a Hamilton-Jacobi equation to construct an approximation of the Moreau envelope which yields an approximation of the proximal operator.
Finally let us mention again the learning approaches, which parametrize the proximal operator as a neural network to learn it (see for instance \cite{hurault2022proximal,hurault2023relaxed,hurault2023convergent}). Also note that inexactness was studied in higher order proximal methods \cite{nesterov2020inexact,nesterov2020inexact2}. 

Now, to show convergence of inexact proximal algorithms, several approaches are possible: ensuring that the norms of the errors are summable \cite{combettes2005signal,condat2013primal,bauschke2017,bonettini2020convergence}, or summable to some power of the iteration count $k$ \cite{schmidt2011convergence,villa2013accelerated,aujol2015stability,barre2023principled,rasch2020inexact}; ensuring that errors are bounded by a quantity proportional to the norm of the difference between iterates \cite{solodov2000error,khanh2025inexact,monteiro2010convergence,alves2020relative,bello2020inexact,alves2020inexact} ("relative error criterion"), which indirectly imposes the norm of the errors to be summable \cite{attouch2010proximal,attouch2013convergence,bolte2014proximal,bello2020inexact} for convergence to occur. Adjacently, proximal algorithms have been studied in the presence of random perturbations \cite{aujol2015stability,combettes2015stochastic,ryu2014stochastic,kim2022convergence} and almost sure convergence is shown provided that the norm of the error is zero in expectation and bounded in variance. For the rest of this paper, we will focus on deterministic assumptions.

If the summability of the norm of the errors cannot be guaranteed, then the previous convergence analysis are no longer valid.
Moreover, if the errors norms are controlled during the optimization, it is often only partially done due to the use of indirect criteria \cite{machart2012optimal,schmidt2011convergence}. In addition, these criteria rely on the knowledge of the underlying function $\phi$. For instance in plug and play optimization \cite{hurault2021gradient,hurault2022proximal,hurault2023convergent} or in unfolded optimization \cite{gregor2010learning}, the proximal operator is learned, therefore we cannot access the true $\phi$, and thus check the summability of potential errors.

This situation calls for the study of weaker assumptions such as boundedness of the errors \cite{boikanyo2013strong,zaslavski2011maximal,hamadouche2022probabilistic,hamadouche2024sharper}. In this case, to the best of our knowledge, current results for proximal point algorithms show the convergence to a solution if the errors are controlled by a vanishing step size \cite{boikanyo2013strong}. For proximal gradient algorithms, the convergence of function values to a value depending on accumulated errors \cite{hamadouche2022probabilistic,hamadouche2024sharper}, and the convergence to a ball around approximate solutions in \cite{sra2012scalable} (which is not convergence to an approximate solution), assuming errors only on the gradient term.

The study we propose here looks at inexact proximal operator as operator on their own in order to better understand the optimization dynamics when they replace exact proximal operators, and the error is either unknown or uncontrolled. In the line of \cite{gribonval2020characterization}, where proximal operators were characterized as the gradient/subgradient of a convex potential $\psi$, we propose \textbf{characterizations of inexact proximal operators}.
\paragraph{Contributions.} In this paper, we present an answer to the question %
\begin{quote}
    \textit{
Let $\mathrm{g}$ be a mapping that approximates $\prox_\phi$. In which sense is $\mathrm{g}$ a good approximation?
    }
\end{quote}
Our first contribution is to make an inventory of approximations, and establish their strengths and weaknesses by proposing criteria they should satisfy. These contributions are summarized in Table \ref{tab:summary_approx}. Our second contribution is a by-product of this exhaustive analysis. We highlight sufficient assumptions for proximal algorithms -- proximal point, proximal gradient descent, Peaceman-Rachford and Douglas-Rachford algorithms -- to converge to an approximate solution of the original problem when approximation errors are neither summable nor vanishing. In particular, we relax the assumption on the contractivity of the proximal operator when $f$ is strongly convex.

We focus in the rest of this work on $\rho$-weakly convex functions $\phi$ with $\rho<1$. 
With this assumption, we study six different ways of approximating $\prox_\phi$: %
\begin{enumerate}[label=(\alph*)]
    \item $\g_a(y) = \prox_{\phi}(y) + e(y)$ and $\Vert e \Vert_\infty \leq \epsilon$; 
    \item $\g_b(y) = \prox_\phi(y+r(y))$, $\Vert r(y) \Vert_\infty \leq \epsilon$; 
    \item $y- \g_c(y) \in \partial_\epsilon^\rho \phi (\g_c(y))$;
    \item $\g_d(y) = \nabla \psi_\epsilon(y)$ with $\psi_\epsilon$ $\epsilon-$close to $\psi$; 
    \item $\g_e(y) \in \partial_\epsilon \psi(y)$;
    \item $\g_f(y) = y - \epsilon \nabla u^\epsilon(y)$ with $u^\epsilon$ $\epsilon$-close to $u$ (the Moreau envelope of $\phi$).
\end{enumerate}
The first three ((a),(b), and (c)) are commonly found in the convex literature \cite{rockafellar1976monotone,solodov2000error,schmidt2011convergence,salzo2012inexact,villa2013accelerated,barre2022note,barre2023principled}. Approximations of type (b) can be equivalently formulated (for convex penalties) as finding an approximate solution to the proximal problem and was introduced as such in \cite{rockafellar1976monotone,guler1992new}. Approximations of type (c) were first introduced in \cite{alber1997proximal}, and refined in \cite{salzo2012inexact} for convex penalties; in \cite{van2024weak,khanh2025inexact} for weakly convex penalties. The fourth one (d) was presented in the plug and play literature, in the sense that a function $\psi$ is parametrized by a differentiable neural network so that taking its gradient would yield a proximal operator (see for instance \cite{hurault2022proximal,fang2023s}). The fifth one (e) is new to the best of our knowledge but seems natural given the characterization of (possibly non continuous) proximal operators in \cite{moreau1965proximite,gribonval2020characterization}. The last one (f) is obtained by solving a particular Hamilton-Jacobi equation and was studied only recently in \cite{osher2023hamilton,heaton2024global}, but was in fact present in the literature since the 1980s \cite{crandall1984two}, without the explicit connection to the proximal operator.

To establish the quality of each type of approximation, we propose several criteria that a good approximation should meet. 
\begin{definition} \label{def:qualitative}
    An $\epsilon$-approximation $\g$ of $\prox_\phi$ is said to be \textbf{qualitative} if there exists a function $\sigma:\RR^+ \to \RR^+$ such that for all $x \in \dom \phi$, $\Vert \prox_\phi(x) - \g(x) \Vert \leq \sigma(\epsilon)$.
\end{definition}
The function $\sigma$ effectively controls how close $\g$ is from $\prox_\phi$ for a fixed $\epsilon$. The next criterion states that $\g$ needs to possess fixed points around local minimizers of $\phi$. 
\begin{definition} \label{def:admissible}
     An $\epsilon$-approximation $\g$ of $\prox_\phi$ is said to be \textbf{admissible} if there exists a fixed point of $\g$ in a neighborhood of a local minimizer $\phi$. 
\end{definition}
Finally, in order to reach these fixed points, the approximation should have some regularity.
\begin{definition}\label{def:regularity}
    An $\epsilon$-approximation $\g$ of $\prox_\phi$ is said to be \textbf{$(L_\g,\gamma)$-Lipschitz} if there exists $L_\g>0, \gamma \geq 0$ such that 
    \begin{equation*}
    (\forall x,y \in \Hi), \quad \Vert \g(x) - \g(y) \Vert \leq L_\g\Vert x-y \Vert + \gamma.
    \end{equation*}
\end{definition}
We add some slack $\gamma$ on the Lipschitz continuity in order to accommodate potential non-zero errors.
We show for each type of approximation, how they satisfy Definitions \ref{def:qualitative} and \ref{def:regularity}. In particular, w.r.t. Definition \ref{def:regularity}, we show that approximations of type (d) have a lower bound on their Lipschitz continuity constant that is related to $L_\psi$ if one wants arbitrary precision $\epsilon$ (Theorem \ref{th:Lipschitzness_lower_bound}). 

Then, to satisfy Definition \ref{def:admissible}, we discuss the existence of fixed points of each one of these approximations. Without further assumptions (e.g., local contractiveness), we can show the existence of fixed points, in general, for approximations of type (c), (e), and (f), albeit on stricter assumptions for approximations of type (e). For the other approximations, the nature of the approximation is not enough, and more information about the problem is necessary. For instance, it is easy to see that a type (a) approximation of the proximal operator of a constant function cannot have fixed points.
Definition \ref{def:admissible} is therefore not trivial to meet and we show for some convex penalties $\phi$ that it can be met, provided that the error is not too big (some of them can be found in Appendix \ref{app:admissible_examples}). Unfortunately, we don't know yet how to guarantee that an approximation is admissible in general, unless they are approximations of type (c), (e) and (f).

Note that the distinction between local minimizers of $\phi$ and fixed points of $\prox_\phi$ is important, otherwise admissibility would be completely vacuous for some non-convex potentials (e.g., MCP \cite{zhang2010nearly}, which has flat regions away from its minimizers). 
Finally, if the errors norm are not summable, in some context we can still expect these norms to decrease to $0$, therefore we need to understand the evolution of the fixed points of $\g$ w.r.t. to the value of $\epsilon$, i.e., denoting $\g^\epsilon$ the approximation of precision $\epsilon$ the quantities
\begin{align}
    \lim_{\epsilon \rightarrow 0^+} \g^\epsilon, \text{ and } \bigcap_{\epsilon \geq 0} \Fix \g^\epsilon.
\end{align}
Ideally, these two quantities are equal to $\Fix \prox_\phi$, and we show under which conditions this can happen. For the simplicity of the presentation, unless the value of $\epsilon$ is important, we will continue to drop it in the notation of the approximations in the following.

As we aim at incorporating this approximation in popular splitting algorithms, the existence of fixed point of the sum or of the composition of $\g$ with other operators is another important matter we need to address. This why we need our approximations to satisfy Definition \ref{def:regularity}.

\noindent We summarize in Table \ref{tab:summary_approx} how each approximation will satisfy Definitions \ref{def:qualitative}, \ref{def:admissible}, and \ref{def:regularity}.
\begin{table}
    \begin{center}
    \begin{tabular}{|c|c|c|c|} \hline
        & \textbf{Quality} $\sigma(\epsilon)$  & \textbf{Regularity} & \textbf{Admissibility}  \\  \hline
       (a) & $\epsilon$ (Eq. \eqref{eq:sigma_typea}) & $(L_\psi,2\epsilon)$ (Prop. \ref{prop:lipschitz_a}) & Problem specific  \\ \hline
       (b) & $L_\psi \epsilon$ (Prop. \ref{prop:sigma_b}) & $(L_\psi,2\epsilon)$ (Prop. \ref{prop:lipschitz_b})  & Problem specific\\ \hline
       (c)  & $\sqrt{L_\psi\epsilon}$ (Prop. \ref{prop:sigma_c}) & $(L_\psi,\sqrt{2L_\psi\epsilon})$ (Prop. \ref{prop:lipschitz_c})  & Yes for convex $\phi$ (Prop. \ref{prop:admissibility_c})\\ \hline
       (d)  & $2\sqrt{L_\epsilon\epsilon}$ (Prop. \ref{prop:sigma_d}) & $(L_{\epsilon},0)$ (Th. \ref{th:Lipschitzness_lower_bound})& Problem specific\\ \hline
       (e)  & $\sqrt{2 L_\psi \epsilon}$ (Prop. \ref{prop:sigma_e}) & $(L_\psi,\sqrt{2L_\psi\epsilon})$ (Prop. \ref{prop:lipschitz_e}) & Yes for convex $\phi$ \\\hline
       (f)  & $\sqrt{N (\lambda^{-1} - \rho)^{-1} \epsilon}$ (Prop. \ref{prop:sigma_f}) & $(L_\psi,0)$ (Prop. \ref{prop:lipschitz_f})  & Yes for convex $\phi$ (Prop. \ref{prop:admissibility_f}) \\ \hline
    \end{tabular}
\end{center}
\caption{\label{tab:summary_approx} Summary of the properties of the approximation: the quality of the approximation measured by $ \Vert g - \prox_\phi \Vert_2  \leq \sigma(\epsilon)$, their $(L_\g,\gamma)$-Lipschitzness and finally their admissibility (i.e., existence of fixed points). $L_\psi$ is the Lipschitz constant of $\prox_\phi$, and $\rho$ the weak convexity modulus. $L_\epsilon$ is related to $L_\psi$ in Theorem \ref{th:Lipschitzness_lower_bound}. For the approximation of type (f), $\lambda$ is the Moreau envelope parameter.}
\end{table}

\paragraph{Notations.}
We denote by $\overline{\RR} = \RR \cup \{+\infty\}$. A function $\phi:\Hi\to\overline{\RR}$ is said to be proper if $\dom \phi \neq \emptyset$ and for all $x\in \Hi$ $\phi(x)>-\infty$. A mapping $\g:\Hi \to \Hi$ is said Lipschitz continuous when there exists $L_\g>0$ such that for all $x,y\in\Hi$, $\Vert \g(x)-\g(y) \Vert \leq  L_\g \Vert x -y \Vert$. The set of continuously differentiable function from $\Hi$ to $\overline{\RR}$ is denoted $C^1(\Hi)$, and we will note $C^{1,1}(\Hi)$ the set of Lipschitz smooth function, i.e., the set of continuously differentiable function with Lipschitz continuous gradient. $\mathbb{B}(x,r)$ denotes the closed ball of center $x\in\Hi$ and radius $r>0$. We note the distance of $x\in \Hi$, to a closed set $S$, $\mathrm{dist}(x,S):=\inf_{y\in S}\Vert x-y \Vert$. We say that $\g$ is $\beta$-cocoercive with $\beta>0$ if for all $x,y \in \Hi$, $\langle \g(x) - \g(y), x-y \rangle\geq \beta \Vert \g(x) - \g(y) \Vert^2$. A proper l.s.c. function $\phi:\Hi \to \overline{\RR}$ is called convex if for all $x,y \in \Hi$ and $t\in[0,1]$, $\phi(tx+(1-t)y) \leq t \phi(x) + (1-t)\phi(y)$. It is called $\mu$-strongly convex if for some $\mu>0$, $\phi - \frac{\mu}{2}\Vert \cdot \Vert^2$ is convex. And finally, it is called $\rho$-weakly convex if for somme $\rho>0$, $\phi + \frac{\rho}{2} \Vert \cdot \Vert^2$ is convex.

\section{Optimization background}
In this section, we recall standard definitions and results in convex and non convex optimization. In particular, we detail important results about the proximal operator in both settings. 
\subsection{Subdifferential and approximate subdifferential}
\begin{definition}{\textbf{Subdifferential \cite{VarAnalRockafellar}.}}
    Let $\phi:\Hi \mapsto \RR$, and let $x\in\Hi$. The Fréchet subdifferential of $\phi$ at $x$ is denoted by $\hat{\partial} \phi(x)$ and is given by the set
    \begin{align*}
        \hat{\partial} \phi(x) = \left\{ s \in \Hi \;|\; \lim_{y\rightarrow x} \inf_{y \neq x} \frac{1}{\Vert x-y \Vert}\left(\phi(y)-\phi(x)-\langle y-x, s \rangle \right) \geq 0 \right\}.
    \end{align*}
    If $x \notin \dom \phi$, then $\hat{\partial} \phi(x) = \emptyset$.
    The limiting subdifferential of $\phi$ at $x$ is denoted by $\partial \phi(x)$ and is given by
    \begin{align*}
\partial \phi(x) = \big\{ v \in \Hi \;|\; \exists & \left( x^k,s^k \right) \overset{k\rightarrow\infty}{\rightarrow} \left(x,v\right) \\ & \textrm{ such that } \phi(x^k) \overset{k\rightarrow\infty}{\rightarrow} \phi(x) \textrm{ and } (\forall k \in \mathbb{N}) ~s^k \in \hat{\partial} \phi(x^k) \big\}.
\end{align*}
Both $\hat{\partial} \phi (x)$ and $\partial \phi(x)$ are closed \cite[Theorem 8.6]{VarAnalRockafellar} and coincide for proper, l.s.c, weakly convex functions \cite[Definition 7.25]{VarAnalRockafellar}. Finally, the Clarke's subdifferential \cite{clarke1975generalized,clarke1990optimization} is defined by
\begin{equation*}
    \partial^0 \phi(x) = \overline{\mathrm{conv}}\partial \phi(x)
\end{equation*}
If $\phi$ is convex the three previous notions coincide, and its subdifferential is given for all $x \in \Hi$ by
\begin{equation}
\label{eq:subdif_convex}
    \partial \phi(x) = \{s \in \Hi \;|\; \phi(x) + \langle s, y-x \rangle \leq \phi(y), \forall y \in \Hi \}.
\end{equation}
\end{definition}
\noindent For convex and weakly convex functions, the subdifferential can be enlarged with a perturbation $\epsilon>0$.
\begin{definition}{\textbf{$\epsilon$-subdifferential for convex functions \cite{brondsted1965subdifferentiability,hiriart-urruty1993convexII}.}}
    Let  $\phi: \Hi \to (-\infty, +\infty]$ be a proper, l.s.c. and convex function. The \textit{enlargement of the subdifferential} of \( \phi \) at \( y \in \Hi \) with parameter \( \epsilon > 0 \) is defined as
    \begin{equation}
        \partial_\epsilon \phi(y) := \left\{ v \in \Hi \mid \phi(x) + \langle v, y-x \rangle - \epsilon \leq \phi(y),~ \forall x \in \Hi \right\}.
    \end{equation}
\end{definition}
\noindent Recognizing that for all $x\in\Hi$ where the Clarke's subdifferential of $\phi$ is non-empty \cite{van2024weak} we have
\begin{equation}
    \partial (\phi + \frac{\rho}{2}\Vert \cdot \Vert^2)(x) = \partial^0 \phi(x) + \rho x,
\end{equation}
one can enlarge the subdifferential of weakly convex function as follows:
\begin{definition}{\textbf{$\rho,\epsilon$-subdifferential for weakly convex functions \cite[Definition 2.10]{khanh2025inexact}.}}\label{def:epsilon_subdiff_weaklyconvex}
    Let $\rho$ and $\epsilon$ be nonnegative numbers. The weak $\epsilon-$subdifferential of a $\rho-$weakly convex function $\phi:\RR^N \mapsto \overline{\RR}$ at $\bar{x}\in \dom \phi$ is defined by 
    \begin{equation}
        \partial_{\epsilon}^\rho \phi(\bar{x}) := \left\{ v \in \RR^N \big| \langle v,x-\bar x \rangle \leq \phi(x) - \phi(\bar x) + \frac{\rho}{2}\Vert x- \bar x \Vert^2 + \epsilon \text{ for all }x \in \RR^N \right\}.
    \end{equation}
\end{definition}
This is equivalent to an $\epsilon$ enlargement of the proximal subdifferential \cite[Definition 8.45]{VarAnalRockafellar} and to the definition of the $\rho,\epsilon$-subdifferential of \cite{van2024weak}. These definitions allow to characterize $\epsilon$-criticality as a critical point $\bar{x}$ of $\phi$ would imply $0 \in \partial \phi(\bar{x})$, and an $\epsilon$-critical point $\bar{z}$ of $\phi$ would imply $0 \in \partial_\epsilon^\rho \phi(\bar{z})$. 
\begin{definition}{\textbf{Subdifferential continuity \cite[Definition 13.28]{VarAnalRockafellar}}.}
A function $\phi : \mathbb{R}^n \to \overline{\mathbb{R}}$ is called 
\emph{subdifferentially continuous} at $\bar{x}$ for $\bar{v}$ if 
$\bar{v} \in \partial \phi(\bar{x})$ and, whenever $(x^\nu, v^\nu) \to 
(\bar{x}, \bar{v})$ with $v^\nu \in \partial \phi(x^\nu)$, one has 
$\phi(x^\nu) \to \phi(\bar{x})$. If this holds for all 
$\bar{v} \in \partial \phi(\bar{x})$, $\phi$ is said to be subdifferentially 
continuous at $\bar{x}$.
\end{definition}
\subsection{Proximal operators and Moreau envelope}
\paragraph{The proximal operator.} The proximal operator of a function $\phi:\Hi \to \overline{\RR}$ with parameter $\lambda>0$ is defined as
\begin{equation}
    \prox_{\lambda \phi}(y) \in \argmin_{x \in \Hi} \Phi_\lambda(x) :=  \phi(x) + \frac{1}{2\lambda} \Vert x - y \Vert^2.
\end{equation}
For a proper, l.s.c. and convex function, this operator is single-valued. Functions whose proximal operator is non-empty for some $\lambda>0$ are called prox-bounded. This prox-boundedness is characterized through the Moreau envelope of $\phi$. For a proper, l.s.c. function $\phi$, the Moreau envelope of $\phi$ with parameter $\lambda>0$ is defined as \cite[Definition 1.22]{VarAnalRockafellar}
\begin{equation}
    (\forall y \in \Hi), \quad u_\lambda(y):=\inf_{x \in \Hi} \phi(x) + \frac{1}{2\lambda} \Vert x-y \Vert^2
\end{equation}
\begin{definition}{Prox-boundedness \cite[Definition 1.23]{VarAnalRockafellar}.}
    A function $\phi : \Hi \mapsto (-\infty,+\infty]$ is said to be prox-bounded if there exists $\lambda > 0$ such that $u_\lambda >-\infty$ for some $y\in \Hi$. The supremum of the set of all such $\lambda$ is the threshold $\lambda_\phi$ of prox-boundedness for $\phi$.
\end{definition}
The first order optimality conditions for the minimization problem associated with the proximal operator yields 
\begin{equation} \label{eq:first_order_prox}
    z \in \prox_{\lambda \phi}(y) \implies 0 \in \partial \Phi_\lambda(z) \Leftrightarrow \frac{y-z}{\lambda} \in \partial \phi(z)
\end{equation}
The three statements are equivalent for convex functions.
For the purpose of clarity, we will follow \cite{gribonval2020characterization} and call $z$ \textit{a} proximal operator of $\phi$, according to the definition:
\begin{definition}
    Let $\mathcal{Y} \subset \Hi$ be non-empty. A function $z:\mathcal{Y}\to\Hi$ is \textit{a} proximal operator of a function $\phi:\Hi \to \overline{\RR}$ if, and only if, $z(y)\in\prox_{\lambda\phi}(y)$.
\end{definition}
Equation \ref{eq:first_order_prox} provides a first characterization of a proximal operator through a subdifferential inclusion. In their article, authors of \cite{gribonval2020characterization} extended the work of Moreau \cite{moreau1962CR,moreau1963CR,moreau1965proximite} and provided a characterization of proximal mappings of (potentially non-convex) penalties as inclusions in the subdifferential of convex potential. First for convex penalties
\begin{proposition}{\cite[Corollary 10.c]{moreau1965proximite}, \cite[Proposition 1]{gribonval2020characterization}.}
    A function \( z : \Hi \to \Hi \) defined everywhere is the proximal operator of a proper convex l.s.c.\ function \( \phi : \Hi \to \overline{\RR} \) if, and only if, the following conditions hold jointly:
\begin{itemize}
    \item[(a)] there exists a convex l.s.c. function \( \psi \) such that for each \( y \in \Hi \), \( z(y) \in \partial \psi(y) \);
    \item[(b)] \( z \) is nonexpansive, i.e.,
    \begin{equation}
        \Vert z(y) - z(y') \Vert \leq \Vert y - y' \Vert, \quad \forall y, y' \in \Hi.
    \end{equation}
\end{itemize}
\end{proposition}
Finally, by relaxing the non-expansiveness of the proximal operator, a similar characterization holds for non-convex penalties.
\begin{theorem}{\cite[Theorem 1]{gribonval2020characterization}}
    Let $\mathcal{Y}\subset\Hi$ be non-empty. A function $z:\mathcal{Y}\to\Hi$ is a proximal operator of a function $\phi:\Hi \to \overline{\RR}$ if, and only if, there exists a convex l.s.c function $\psi:\Hi\to \overline{\RR}$ such that for every $y \in \mathcal{Y}$, $z(y) \in \partial \psi(y)$. 
\end{theorem}
This result is a consequence of a more general result \cite[Theorem 3]{gribonval2020characterization}, extending for instance these characterizations to Bregman proximal operators \cite{censor1992proximal}. 
The regularity of the proximal operator has a direct connection to the regularity of $\psi$: $z$ is continuous on $\mathcal{Y}$ non-empty and open, if, and only if $\psi$ is continuously differentiable on $\mathcal{Y}$ \cite[Corollary 1]{gribonval2020characterization}. 
Moreover, $z$ is continuously differentiable on $\mathcal{Y}$ non-empty, open and convex if, and only if, $\psi$ is twice continuously differentiable on $\mathcal{Y}$ \cite[Corollary 6]{gribonval2020characterization}.
Also, if the proximal operator is $L>0$-Lipschitz continuous then $\phi$ is $(1-1/L)$-weakly convex:
\begin{theorem}{\cite[Proposition 2]{gribonval2020characterization} \cite[Theorem 3.5]{luo2024various}.}
    Let $\phi:\Hi \to \overline{\RR}$ be a proper l.s.c. function with prox-bound $\lambda_\phi>0$, let $\lambda \in (0,\lambda_\phi)$, and let $L \in (0,+\infty)$. Then, the following are equivalent:
    \begin{enumerate}
        \item $\prox_{\lambda \phi}$ is $L-$Lipschitz.
        \item $\lambda \phi$ is $(1/L-1)$ strongly convex, i.e., $\lambda \phi-(1/L-1)\Vert \cdot \Vert^2$ is convex.
    \end{enumerate}
    If $\phi$ is $\mu-$strongly convex then $\prox_\phi$ is a strict contraction, i.e., $L = \frac{1}{1+\mu}$ \cite{bauschke2012firmly}.
\end{theorem}
The Lipschitz continuity of $\prox_\phi$ can be equivalently expressed with the constant of weak convexity of $\phi$ as $L = \frac{1}{1-\rho}$, which imposes $\rho<1$.
Finally, we can make a connection between $\psi$ and the Moreau envelope of $\phi$ of parameter $1$, there exists $C\in\RR$ such that for all $x\in \Hi$
\begin{equation*}
    \psi(x) = u(x) - \frac{1}{2}\Vert x \Vert^2 + C.
\end{equation*}

The largest class of functions whose proximal operator is single-valued at $x\in\Hi$ is the class of prox-regular functions:
\begin{definition}{\textbf{Prox-regularity of functions} \cite[Definition 13.27]{VarAnalRockafellar}} A function $\phi:\RR^N \mapsto \overline{\RR}$ is prox-regular at $\bar x$ for $\bar v$ if $\phi$ is finite and locally l.s.c. at $\bar x$ with $\bar v \in \partial \phi(\bar x)$, and there exists $\epsilon >0$ and $\rho \geq 0$ such that
\begin{align}
    \phi(x') \geq \phi(x) + \langle v, x'-x \rangle - \frac{\rho}{2} |x'-x|^2 \text{ for all } x' \in \mathbb B(\bar x, \epsilon) \\
    \text{when } v\in \partial \phi(x), | v- \bar v | < \epsilon, |x-\bar x | < \epsilon, \phi(x) < \phi(\bar x) + \epsilon.
\end{align}
When this holds for all $\bar v \in \partial \phi(\bar x)$, $\phi$ is said to be prox-regular at $\bar x$.
\end{definition}
\noindent Weakly convex functions are prox-regular everywhere for the proper scaling.
\subsection{Error bounds}
One can characterize the behavior of a function around its set of minimizers using error bounds. We present some of them for convex and weakly convex functions.
\begin{theorem}{\cite[Theorem 3]{bolte2017error}} \label{th:error_bound_convex}
    Let $\phi$ be a proper, l.s.c., convex, and semi-algebraic function. Moreover $\argmin \phi$ is nonempty and compact. Then $\phi$ has a global error bound
    \begin{equation}
        (\forall x \in \Hi), \quad (\phi(x)-\min_{x\in\Hi}\phi(x))^{1+\frac{1}{p}} \geq \gamma_0 \mathrm{dist}\left(x,\argmin \phi\right),
    \end{equation}
    where $\gamma_0>0$ and $p\geq1$ is a rational number.
\end{theorem}
We will also need the concept of metric sub-regularity:
\begin{definition}{Metric sub-regularity \cite{dontchev2009implicit}} A set-valued mapping $T: \Hi \mapsto 2^{\Hi}$ is called metrically sub-regular at $\bar{z}\in\Hi$ for $\bar{u}\in T(\bar{z})$ if there exists $\kappa\geq 0$ along with a neighborhood $\mathcal{Z}$ of $\bar{z}$ such that
    \begin{equation}
        (\forall z \in \mathcal{Z}), \quad \mathrm{dist}(z, T^{-1}(\bar{u})) \leq \kappa \mathrm{dist}(\bar u, T(z)).
    \end{equation}
\end{definition}
\subsection{Convergence of fixed point iterations}
We conlude this background section by presenting the necessary tools for the convergence analysis of fixed point iterations \cite{bauschke2017,liang2016convergence,cortild2025krasnoselskii,arakcheev2025opial}.
\paragraph{Convex setting}
Recall that a single valued operator \(T:\RR^N \to \RR^N\) is called monotone if \cite[Definition 20.1]{bauschke2017} 
\begin{equation}
\left(\forall x\in \RR^N\right), \left(\forall y\in\RR^N\right), 
\quad \langle x - y, Tx -  Ty \rangle \geq 0.
\end{equation}
It is called $\alpha$-averaged, with $\alpha \in (0,1)$, if $T$ is nonexpansive and there exists a nonexpansive operator $S:\RR^N \to \RR^N$ such that $T= (1-\alpha)\Id + \alpha S$ \cite[Definition 4.33]{bauschke2017}. The reflection of an operator is denoted as $R_T:=2T- \Id$ \cite{bauschke2017}. 
\begin{theorem}{Baillon-Haddad \cite[18.17]{bauschke2017}}
    Let $\psi:\RR^N \to \RR$ be a Fréchet differentiable convex function and $L>0$. Then $\nabla \psi$ is $L$-Lipschitz continuous if and only if $\nabla \psi$ is $1/L$-cocoercive.
\end{theorem}
It follows from \cite[Proposition 4.31(iii)]{bauschke2017} that $\nabla \psi$ is $1/L$-cocoercive if and only if $\frac{1}{L} \nabla \psi$ is $1/2$-averaged. 
The classical framework of convergence for averaged operators is that of Krasnosel'ski\u{\i}--Mann iterations.
\begin{proposition}{\cite[Proposition 5.16]{bauschke2017}}
Let $\alpha \in (0,1)$. Let $T:\RR^N \mapsto \RR^N$ be an $\alpha$-averaged operator such that $\Fix T \neq \emptyset$. Let $(\lambda_n)_{n\in\mathbb{N}}$ be a sequence in $[0,1/\alpha]$ such that $\sum_{n\in \mathbb{N}} \lambda_n(1-\alpha\lambda_n) = + \infty$, and let $x_0 \in \RR^N$. Set 
\begin{equation}
    (\forall n\in \mathbb{N}), \quad x_{n+1} = x_n + \lambda_n \left(T(x_n) - x_n \right).
\end{equation} 
Then the following hold:
\begin{enumerate}[label=(\roman*)]
    \item $(x_n)_{n\in\mathbb{N}}$ is Fejér monotone with respect to $\Fix T$
    \item $\left( T(x_n) - x_n \right)_{n\in\mathbb{N}}$ converges strongly to $0$.
    \item $(x_n)_{n \in \mathbb{N}}$ converges weakly (and thus strongly here) to a point in $\Fix T$.
\end{enumerate}
\end{proposition}

We say that an operator $T$ is quasi nonexpansive if $\Fix T \neq \emptyset$, and $\Vert Tx-p \Vert \leq \Vert x - p \Vert$ for all $x\in \Hi$ and $p \in \Fix T$ \cite{cortild2025krasnoselskii}.
\begin{definition}
Let $T : \Hi \to \Hi$ be an operator on a Hilbert space $H$.  
The operator $\Id - T$ is said to be \emph{demiclosed at $0$} if
\begin{equation*}
x_n \rightharpoonup x, \quad (\Id - T)x_n \to 0 
\quad \Longrightarrow \quad x \in \operatorname{Fix}(T).
\end{equation*}
\end{definition}
A family $(T_k)$ of operators is asymptotically demiclosed at $0$ if for every sequence $(u_k) \in \Hi$, such that $u_k \rightharpoonup u$ and $T_k u_k -u_k \rightarrow 0$, it follows that $u \in \bigcap_{k\geq 1} \Fix T_k$.
For quasi nonexpansive operators, the convergence to fixed point was studied for sequences generated by the following algorithm in \cite{cortild2025krasnoselskii} as a generalization of Krasnosel'ski\u{\i}--Mann iterations:
\begin{align}
    y_k & = x_k + \alpha_k (x_k - x_{k-1}) + \epsilon_k \nonumber\\
    z_k &= x_k + \beta_k (x_k - x_{k-1}) + \rho_k \nonumber\\
    x_{k+1} &= (1-\lambda_k)y_k + \lambda_k T_k z_k + \theta_k \label{eq:km_peypouquet}
\end{align}
where $T_k : \Hi \to \Hi$, $\alpha_k$, $\beta_k$, $\lambda_k \in [0,1]$, and $\epsilon_k,\rho_k,\theta_k \in \Hi$ for $k\geq 1$ and $x_0,x_1 \in \Hi$ (set equal). Define also for all $k$, $\mu_k := (1-\lambda_k)\alpha_k + \lambda_k \beta_k$. Authors of \cite{cortild2025krasnoselskii} showed the convergence of $(x_k,y_k,z_k)$ to a fixed point of $T$ under the following assumptions on the parameters:
\begin{assumption} \label{ass:km}
    The sequences $(\alpha_k)$, $(\beta_k)$, $(\lambda_k)$ and $(\mu_k)$, along with the constraints $\alpha = \inf \alpha_k$, $A = \sup \alpha_k$, $\lambda = \inf \lambda_k,$ $\Lambda =\sup \lambda_k$ and $M = \sup \mu_k$, satisfy: $\alpha,A,M \in [0,1)$, $\lambda, \Lambda \in (0,1)$, $0 \leq \beta_k \leq 1$, and $\mu_k \leq \mu_{k+1}$ for all $k\geq 1$. 
    Moreover, parameters satisfy the following compatibility condition
    \begin{equation}
        \sup_{k\geq 1} \left[ (1-\lambda_k) \alpha_k(1+\alpha_k) + \lambda_k \beta_k(1+\beta_k) + (\lambda_k^{-1}-1) \alpha_k (1-\alpha_k) - (\lambda_{k-1}^{-1}-1)(1-\alpha_{k-1})\right] \geq 0.
    \end{equation}
\end{assumption}
\begin{theorem} \label{th:peypouquet}Let $T_k:\Hi \to \Hi$ be a family of quasi-nonexpansive operators such that $F:= \bigcap_{k\geq 1} \Fix T_k \neq \emptyset$, let Assumption \ref{ass:km} hold, and suppose that the error sequences $(\epsilon_k),(\rho_k)$ and $(\theta_k)$ are summable. Assume, moreover that the family $(\Id- T_k)_k$ is asymptotically demiclosed at $0$. If $(x_k,y_k,z_k)$ is generated by algorithm \eqref{eq:km_peypouquet}, then $(x_k,y_k,z_k)$ converges to $(p^*,p^*,p^*)$ with $p^* \in F$.
\end{theorem}
\noindent If $T$ is a contraction then it trivially satisfies the assumptions of this theorem.
\paragraph{Non-convex setting.}
We conclude this reminder by a local convergence result which we can use if the approximations are only locally contractive around their fixed points. The assumptions of this theorem cannot be verified without constructing the approximations explicitly, thus it has an information purpose only.
\begin{theorem}{Local convergence of fixed point iteration\cite{stepleman1975characterization,ostrowski1960solutions,ortega1966nonlinear}.}
Let $T:\RR^N \mapsto \RR^N$ be differentiable at a point $x^*\in\RR^N$. Suppose that $x^*$ is a fixed point of $T$. Let $J_T(x^*)$ denote the Jacobian of $T$ at $x^*$, and denote by $\rho(J_T(x^*))$ its spectral radius. If $\rho(J_T(x^*))<1$, then there exists a neighborhood $U$ of $x^*$, such that for any $x_0 \in U$, the sequence defined by the iterations
\begin{equation}
    (\forall n \in \mathbb{N}), \quad x_{n+1} = T(x_n)
\end{equation}
converges to $x^*$.
\end{theorem}

\section{Approximations of the proximal operator}
In this section, we present characterizations of inexact proximal operators, with as much detail as possible. %
In the following, we will assume that
\begin{assumption} \label{ass:weaklyconvex}
    Let $L_\psi>0$. $\phi$ is a proper, l.s.c., $\rho = \left(1-\frac{1}{L_\psi}\right)$-weakly convex function with prox-bound $\lambda_\phi>1$.
\end{assumption}
This assumption is mildly restrictive: if $\rho\geq 1$, then $\prox_{\gamma \phi}$ is Lipschitz continuous for $\gamma<1/\rho$.%
\subsection{Approximation of type (a)}
The approximation of type (a) is the simplest one, and is defined by adding an additive error outside the computation of the proximal operator.
\begin{definition}{\textbf{Type (a) approximation}.} \label{def:approx_typea}
    We say that $\g_a$ is an $\epsilon$-type (a) approximation of $\prox_\phi$ if for all $y\in\Hi$ there exists $e(y)$ such that $\Vert e(y) \Vert \leq \epsilon$ and,
    \begin{equation}
        \g_a(y) = \prox_\phi(y) + e(y).
    \end{equation}
\end{definition}
Approximations of type (a) verify easily Definition \ref{def:qualitative}. Indeed, Definition \ref{def:approx_typea} implies directly that
\begin{equation} \label{eq:sigma_typea}
    (\forall y \in \RR^N), \quad \Vert \g_a(y) - \prox_\phi(y) \Vert = \Vert e(y) \Vert \leq \epsilon.
\end{equation}
Commonly in the convex optimization literature, $\g_a$ is written directly this way, for instance in \cite{salzo2012inexact,villa2013accelerated}, but with a square root on $\epsilon$.
Such bound can be obtained when $\g(y)$ is a type 1 approximation \cite{salzo2012inexact,villa2013accelerated} of $\prox_\phi(y)$, where $\phi$ is a proper, l.s.c., convex potential \cite{salzo2012inexact} and the error is measured in $\ell_2$-norm. We reproduce the proof: $\Phi$ is $1$-strongly convex and $0\in \partial \Phi \left(\prox_\phi(y)\right)$, thus
\begin{equation}
    \Phi(\g(y)) - \Phi(\prox_\phi(y)) \geq \frac{1}{2} \Vert \g(y) - \prox_\phi(y) \Vert^2.
\end{equation}
Then using the $\epsilon$-solution properties, we get 
\begin{equation}
    \Phi(\g(y))-\Phi(\prox_\phi(y))\leq \epsilon,
\end{equation}
which yields the desired result.

\paragraph{Lipschitz continuity of $\g_a$}
Without surprise, the Lipschitz continuity of $\g_a$ is that of $\prox_\phi$.
\begin{proposition} \label{prop:lipschitz_a}
    Suppose that Assumption \ref{ass:weaklyconvex} holds. Let $\epsilon>0$  and let $\g_a$ be an $\epsilon$-type (a) approximation of $\prox_\phi$. Then,
    \begin{equation}
        (\forall x,y \in \Hi), \quad \Vert \g_a(x)-\g_a(y) \Vert \leq L_\psi \Vert x-y \Vert + 2\epsilon.
    \end{equation}
\end{proposition}
\begin{proof}
    Straightforward.
\end{proof}

\paragraph{Admissibility of $\g_a$.} The existence of fixed points is not guaranteed unless we assume further structure on the additive error $e$. It is reasonable to assume that they do, and in this context we can characterize them w.r.t. to those of $\prox_\phi$. For these results, we have to assume that $\phi$ is convex and coercive.
\begin{proposition}Suppose that $\phi$ is proper, l.s.c., convex and coercive. Suppose also that there exists $\overline{\epsilon}>0$ such that for all $\epsilon \leq \overline{\epsilon}$, $\g_a^\epsilon$ admits fixed points, and that for every sequence $\epsilon_n \downarrow 0$, and every choice $y_n \in \Fix \g^\epsilon_a$, $(y_n)$ is bounded. Then,
    \begin{equation}
        \limsup_{\epsilon \rightarrow 0^+} \Fix \g^\epsilon_a \subseteq \Fix \prox_\phi.
    \end{equation}
\end{proposition}
\begin{proof}
    Set $\epsilon_n \downarrow 0^+$, and $y_n \in \Fix \g^{\epsilon_n}_a$. By assumption, there exists a subsequence $y_{n_k} \rightarrow_{n_k\rightarrow+\infty} y \in \Hi$.
    We have, using Lipschitz continuity of $\prox_\phi$ and the definition of $\g_a$, that
    \begin{align*}
        \Vert y - \prox_\phi(y) \Vert & = \Vert y -y_{n_k} + y_{n_k} - \prox_\phi(y_{n_k}) + \prox_\phi(y_{n_k}) - \prox_\phi(y) \Vert \\
        & \leq 2 \Vert y - y_{n_k} \Vert + \epsilon_{n_k} \\
        & \rightarrow 0 \text{ as } n_k \rightarrow + \infty.
    \end{align*}
    Hence,
    \begin{equation*}
        \limsup_{\epsilon \rightarrow 0^+} \Fix \g^\epsilon_a \subseteq \Fix \prox_\phi.
    \end{equation*}
\end{proof}
The assumption of bounded sequence of fixed points is strong, as it may not hold without structure on the additive error or on $\phi$. It can hold for instance if $\phi$ is strongly convex, by
\begin{equation*}
    (\forall y^\epsilon \in \Fix \g^\epsilon_a), \quad \Vert y^\epsilon - y^* \Vert \leq \frac{\epsilon}{1-L_\psi},
\end{equation*}
where $y^*$ is the unique fixed point of $\prox_\phi$. Or if $\prox_\phi-x$ is metrically sub-regular at $0$, i.e. if for a neighborhood $\mathcal{N}$ of $\Fix\prox_\phi$, we had 
\begin{equation*}
    \forall z \in \mathcal{N}, \quad \mathrm{dist}(z,\Fix \prox_\phi) \leq \kappa \Vert z -\prox_\phi(z)\Vert.
\end{equation*}
In this setting, if the fixed points of $\g_a^\epsilon$ are sufficient close to those of $\prox_\phi$ their distance to $\Fix \prox_\phi$ would be bounded by $\kappa \epsilon$.

\subsection{Approximations of type (b)}
The approximation of type (b) is defined by adding a residual error in the proximal inclusion in the subdifferential. The norm of this residual error is controlled by $\epsilon$.
\begin{definition}{\textbf{Type (b) approximation.}} \label{def:approx_typeb} We say that $\g_b$ is an $\epsilon$-type (b) approximation of $\prox_\phi$ if for all $y\in\Hi$ there exists $r(y)$ such that $\Vert r(y) \Vert \leq \epsilon$ and,
    \begin{equation}
        \g_b(y) = \prox_\phi(y + r(y)).
    \end{equation}
\end{definition}
\noindent This characterization implies
\begin{equation}
    \left(y+r(y)\right) -\g_b(y) \in \partial \phi(\g_b(y))
\end{equation}
If $\phi$ is convex, then this is equivalent to saying 
\begin{equation}
    \g_b(y) = \prox_\phi(y+r(y)).
\end{equation} From the definition, we derive immediately the quality of the approximation.
\begin{proposition} \label{prop:sigma_b}
    Suppose that Assumption \ref{ass:weaklyconvex} holds. Let $\g_b:\RR^N \mapsto \RR^N$ be an $\epsilon$-type (b) approximation of $\prox_\phi$. Then,
    \begin{equation}
        \Vert \g_b(x) - \prox_\phi(x) \Vert \leq L_\psi \epsilon.
    \end{equation}
\end{proposition}
\begin{proof}
    Invoke Lipschitz continuity of $\prox_\phi$.
\end{proof}

\paragraph{Lipschitz continuity of $\g_b$.} The regularity of $\g_b$ is fairly straightforward, in the same manner as the regularity of $\g_a$.
\begin{proposition} \label{prop:lipschitz_b}
    Let $\phi$ be a $1-1/L_\psi$-weakly convex function with prox-bound $\lambda_\phi>1$, and $L_\psi>0$. Let $\epsilon>0$  and let $\g_b$ be an $\epsilon$-type (b) approximation of $\prox_\phi$. Then,
    \begin{equation}
        (\forall x,y \in \Hi), \quad \Vert \g_b(x)-\g_b(y) \Vert \leq L_\psi \left(\Vert x-y \Vert + 2\epsilon\right).
    \end{equation}
\end{proposition}
\paragraph{Admissibility of $\g_b$.} The existence of fixed points is not guaranteed unless we assume further structure on the residual error $r$. Again we deem reasonable to assume that they do, and in this context we can characterize them w.r.t. to those of $\prox_\phi$. 
\begin{proposition} \label{prop:typec_subdiff_dist}
    Let $\phi:\RR^N \mapsto \overline{\RR}$ be a proper, l.s.c., convex function. Suppose that there exists fixed points of $\g_b$. Then for all $x\in \Fix \g_b$, we have:
    \begin{equation}
        \mathrm{dist}\left(0,\partial \phi(x) \right) \leq \epsilon
    \end{equation}
\end{proposition}
\begin{proof}
    We have by definition of $\g_b$ that
    \begin{equation*}
        x-\g_b(x) \in \partial \phi(\g_b(x)) - r(x),
    \end{equation*}
    hence
    \begin{equation*}
        r(x) \in \partial \phi(x).
    \end{equation*}
    Also by definition $\Vert r(x) \Vert \leq \epsilon$, which concludes the proof.
\end{proof}
We cannot hope to show that $\mathrm{dist}(\Fix \g_b, \argmin \phi)\leq \epsilon$ without further assumptions. Indeed take for instance $\phi: x \mapsto |x|$. The subdifferential of $\phi$ is:
\begin{equation}
    \partial \phi(x) = \begin{cases}
        \{-1\} & \text{if } x<0, \\
        [-1,1] & \text{if } x = 0, \\
        \{1\} & \text{if } x>0.
    \end{cases}
\end{equation}
Hence, for all $\epsilon>1$, all $x\in\Hi$ are such that $\mathrm{dist}\left(0, \partial \phi(x) \right) \leq \epsilon$. This behavior is a direct consequence of the lack of inner continuity of the subdifferential, i.e., subgradients at critical point cannot be approached by limits of subgradients at nearby points. We have however a similar result as for approximations of type (a).

\begin{proposition}Suppose that $\phi$ is proper, l.s.c., convex and coercive. Suppose also that there exists $\overline{\epsilon}>0$ such that for all $\epsilon \leq \overline{\epsilon}$, $\g_b^\epsilon$ admits fixed points, and that for every sequence $\epsilon_n \downarrow 0$, and every choice $y_n \in \Fix \g^\epsilon_b$, $(y_n)$ is bounded. Then,
    \begin{equation}
        \limsup_{\epsilon \rightarrow 0^+} \Fix \g^\epsilon_b \subseteq \Fix \prox_\phi.
    \end{equation}
\end{proposition}
\begin{proof}
    Set $\epsilon_n \downarrow 0^+$, and $y_n \in \Fix \g^{\epsilon_n}_b$. By assumption, there exists a subsequence $y_{n_k} \rightarrow_{n_k\rightarrow+\infty} y \in \Hi$.
    We have, using Lipschitz continuity of $\prox_\phi$ and the definition of $\g_b$, that
    \begin{align*}
        \Vert y - \prox_\phi(y) \Vert & = \Vert y -y_{n_k} + y_{n_k} - \prox_\phi(y_{n_k}+r_{\epsilon_{n_k}}(y_{n_k})) + \prox_\phi(y_{n_k}+r_{\epsilon_{n_k}}(y_{n_k})) - \prox_\phi(y) \Vert \\
        & \leq 2 \Vert y - y_{n_k} \Vert + \Vert y_{n_k}- \g_b^{\epsilon_{n_k}}(y_{n_k})\Vert + \epsilon_{n_k} \\
        & \rightarrow 0 \text{ as } n_k \rightarrow + \infty.
    \end{align*}
    Hence,
    \begin{equation*}
        \limsup_{\epsilon \rightarrow 0^+} \Fix \g^\epsilon_b \subseteq \Fix \prox_\phi.
    \end{equation*}
\end{proof}

\paragraph{An equivalence with approximations of type (a).} For the class of weakly convex function, the proximal operator is Lipschitz continuous and single-valued. For this class of function, some approximations of type (b) are in fact approximations of type (a)
\begin{lemma}
    Suppose that Assumption \ref{ass:weaklyconvex} holds. Then suppose that we have for all $x \in K$ that $\g_b(x):=\prox_\phi(x+ r(x))$, with $\Vert r(x) \Vert_2 \leq \epsilon$, i.e., $\g_b$ is an $\epsilon$-type (b) approximation of $\prox_\phi$. Then, for all $x \in K$
    \begin{equation}
        \Vert \g_b(x) - \prox_{\phi}(x) \Vert_2 \leq L_\psi \epsilon,
    \end{equation}
    and in particular, for all  $y \in \Fix \g_b$
    \begin{equation}
        \Vert y - \prox_{\phi}(y) \Vert_2 \leq L_\psi \epsilon.
    \end{equation}
\end{lemma}
\begin{proof}
    We have $\Vert \g_b(x) - \prox_\phi(x) \Vert_2 = \Vert \prox_\phi(x + r(x)) - \prox_\phi(x) \Vert_2 \leq L_\psi \Vert r(x)\Vert_2$. Hence, $\Vert \g_b(x) - \prox_\phi(x) \Vert_2 \leq L_\psi \epsilon$. 
\end{proof}
\subsection{Approximations of type (c)}
The approximation of type (c) is constructed by replacing the proximal inclusion in the subdifferential, by the approximation inclusion in the $\epsilon,\rho$-subdifferential, $\epsilon$ measuring the enlargement of the subdifferential.
\begin{definition}{\textbf{Type (c) approximation.}} \label{def:approx_typec}
    Let $\phi:\Hi \mapsto \overline{\RR}$ be a proper, l.s.c., $\rho$-weakly convex function. We say that $\g_c$ is an $\epsilon$-type (c) approximation of $\prox_\phi$ if for all $y \in \Hi$
    \begin{equation}
        y -\g_c(y) \in \partial_\epsilon^\rho \phi(\g_c(y)).
    \end{equation}
\end{definition}
Invoking the everywhere prox-regularity of weakly convex functions, we can quantify the quality of this approximation.
\begin{proposition} \label{prop:sigma_c}
    Suppose that Assumption \ref{ass:weaklyconvex} holds. Let $\g_c$ be an $\epsilon$-type (c) approximation of $\prox_\phi$. The following holds
    \begin{equation}
        (\forall x \in \Hi), \quad \Vert \g_c(x) - \prox_\phi(x) \Vert_2 \leq \sqrt{\frac{\epsilon}{1-\rho}}.
    \end{equation}
\end{proposition}
\begin{proof}
    Weakly convex functions are prox-regular everywhere with the prox-regular constant uniformly equal to $\rho$ \cite{atenas2023unified}. Hence, we have 
    \begin{align*}
        x -\prox_\phi(x) \in \partial \phi(\prox_\phi(x)) \implies (\forall y \in \Hi), \quad &\phi(\prox_\phi(x)) + \langle x-\prox_\phi(x), y - \prox_\phi(x)\rangle \\ &\leq \phi(y) + \frac{\rho}{2}\Vert y - \prox_\phi(x) \Vert_2^2.
    \end{align*}
    Moreover, by definition of $\g_c$, we also have 
    \begin{equation*}
        (\forall y \in \Hi), \quad \phi(\g_c(x)) + \langle x-\g_c(x), y - \g_c(x)\rangle \leq \phi(y) + \frac{\rho}{2}\Vert y - \g_c(x) \Vert_2^2+\epsilon,
    \end{equation*}
    implying that 
    \begin{equation*}
        \Vert \prox_\phi(x) - \g_c(x) \Vert^2_2 \leq \rho \Vert   \prox_\phi(x) - \g_c(x) \Vert^2_2 + \epsilon,
    \end{equation*}
    which concludes the proof.
\end{proof}
\paragraph{Lipschitz continuity of the approximation.}
\begin{proposition} \label{prop:lipschitz_c}
    Suppose that Assumption \ref{ass:weaklyconvex} holds. Let $\epsilon>0$ and $\g_c$ be an $\epsilon$-type (c) approximation of $\prox_\phi$. Then, the following hold
    \begin{equation}
        (\forall x,y \in \Hi), \quad \Vert \g_c(x) - \g_c(y) \Vert \leq \frac{1}{1-\rho}\Vert x-y \Vert + \sqrt{\frac{2\epsilon}{1-\rho}}.
    \end{equation}
\end{proposition}
\begin{proof}
    By definition of the $\rho,\epsilon$-subdifferential, we have that for all $x,y \in \Hi$
    \begin{equation*}
        \langle \g_c(x) - \g_c(y), \g_c(x)-\g_c(y) + (y-x)  \rangle \leq \rho \Vert \g_c(x) - \g_c(y) \Vert^2 + 2\epsilon.
    \end{equation*}
    Hence,
    \begin{align*}
        (1-\rho)\Vert \g_c(x) - \g_c(y) \Vert^2 & \leq 2 \epsilon + \langle \g_c(x) -\g_c(y), x-y \rangle \\
        & \leq 2\epsilon + \Vert \g_c(x) - \g_c(y) \Vert \Vert x-y \Vert.
    \end{align*}
    This is a quadratic inequality in $\Vert \g_c(x) -\g_c(y) \Vert$ that holds if
    \begin{equation*}
        \frac{\Vert x-y \Vert - \sqrt{\Vert x-y \Vert^2 + 8 (1-\rho)\epsilon}}{2(1-\rho)} \leq \Vert \g_c(x) - \g_c(y) \Vert \leq \frac{\Vert x-y \Vert +\sqrt{\Vert x-y \Vert^2 + 8(1-\rho) \epsilon}}{2(1-\rho)}.
    \end{equation*}
    The right hand side can be upper bounded by 
    \begin{align*}
        \frac{\Vert x-y \Vert +\sqrt{\Vert x-y \Vert^2 + 8 (1-\rho)\epsilon}}{2} & \leq \frac{\Vert x-y \Vert +\sqrt{(\Vert x-y \Vert + 2 \sqrt{2(1-\rho)\epsilon})^2}}{2(1-\rho)} \\
        & \leq \frac{1}{1-\rho}\Vert x-y \Vert + \sqrt{\frac{2\epsilon}{1-\rho}},
    \end{align*}
    which concludes the proof.
\end{proof}

\paragraph{Admissibility of $\g_c$.} If $\phi$ is convex, an immediate consequence for an approximation of type (c) of $\prox_\phi$ is that the fixed points of $\prox_\phi$ yield $0$ of $\partial_\epsilon \phi$ for any $\epsilon\geq 0$.
\begin{equation}
    x \in \mathrm{Fix }\g_c \implies 0 \in \partial_\epsilon \phi(\g_c(x)).
\end{equation}
Therefore by the properties of $\partial_\epsilon\phi$ that $\bigcap_{\epsilon\geq 0} \Fix \g_c^\epsilon =\Fix \prox_\phi$, which is a much stronger result than what we could obtain for approximation of type $(a)$ and $(b)$.
\begin{proposition} 
    Let $\phi:K \mapsto \overline{\RR}$ be a proper, l.s.c., convex function. Let $\g_c:K \mapsto K$ be an approximation of type (c) (\ref{def:approx_typed}) of $\prox_\phi$. Then fixed points of $\g_c$ are $\epsilon$-optimal solution of
    \begin{equation}
        \argmin_{x \in \Hi} \phi(x).
    \end{equation}
\end{proposition}
\begin{proof}
    Fixed points of $\g_c$ are, by Definition \ref{def:approx_typec} such that
    \begin{equation*}
        (\forall x \in \mathrm{Fix}~\g_c), \quad x-\g_c(x) = 0 \in \partial_\epsilon \phi(x).
    \end{equation*}
    Thus by \cite[Theorem XI.1.1.5]{hiriart-urruty1993convexII}, we have
    \begin{equation*}
        (\forall y \in \Hi), \quad \phi(x) \leq \phi(y) + \epsilon,
    \end{equation*}
    and the desired result by setting $y \in \argmin \phi$.
\end{proof}

With the convex error bound, we can quantify how much the set of $\epsilon$-critical points grows around the set of critical points.
\begin{proposition} \label{prop:diam_epsilon_subdiff}
    Suppose $\phi$ is a proper, l.s.c., convex, and semi-algebraic function. Moreover $\argmin \phi$ is nonempty and compact. Without loss of generality, suppose also that $\min \phi = 0$. Then there exists $p\geq1$, rational, and $\gamma_0>0$ such that for all $\epsilon>0$
    \begin{equation}
        \diam \left(\left(\partial_\epsilon \phi \right)^{-1}(0)\right) \leq \frac{2}{\gamma_0} \epsilon^{1+\frac{1}{p}}.
    \end{equation}
\end{proposition}
\begin{proof}
    By \cite[Theorem XI.1.1.5]{hiriart-urruty1993convexII},
    \begin{equation*}
        \left(\partial_\epsilon \phi \right)^{-1}(0) = \left\{ x \big| \phi(x) \leq \phi(y) + \epsilon, ~\forall y \in \Hi\right\}.
    \end{equation*}
    In particular, $\phi(x) \leq \epsilon$ when taking  $y=\bar{x} \in \argmin \phi$. Then applying Theorem \ref{th:error_bound_convex} we have
    \begin{equation*}
        \gamma_0 \mathrm{dist}\left(x,\argmin \phi\right) \leq \epsilon^{1+\frac{1}{p}}.
    \end{equation*}
    Now, as $\bar{x} \in \argmin \phi$ trivially implies $\bar{x} \in \left(\partial_\epsilon \phi \right)^{-1}(0)$, we have:
    \begin{equation*}
        \diam \left(\left(\partial_\epsilon \phi \right)^{-1}(0)\right) \leq \frac{2}{\gamma_0} \epsilon^{1+\frac{1}{p}}.
    \end{equation*}
\end{proof}
The next result characterizes how the fixed point of $\g_c$ behave (if they exist) with respect to the desired fixed points of $\prox_\phi$.

\begin{proposition} \label{prop:admissibility_c}
    Let $\phi$ be a proper, l.s.c., convex, and semi-algebraic function. Let $\g_c$ be a continuous type (c) $\epsilon$-approximation of $\prox_\phi$. Then we have:
    \begin{equation}
        \mathrm{dist} \left(\Fix \g_c, \argmin \phi \right) \leq \frac{2}{\gamma_0} \epsilon^{1+\frac{1}{p}}.
    \end{equation}
\end{proposition}
\begin{proof}
    By definition of an approximation of type (c), we have for all $x \in \Fix \g_c$ that
    \begin{equation*}
        0 \in \partial_\epsilon \phi (x).
    \end{equation*}
    Hence, by Proposition \ref{prop:diam_epsilon_subdiff} we have
    \begin{equation*}
        \mathrm{dist} \left(x, \argmin \phi \right) \leq \frac{2}{\gamma_0} \epsilon^{1+\frac{1}{p}},
    \end{equation*}
    which concludes the proof.
\end{proof}

\subsection{Approximations of type (d)}
The approximation of type (d) is defined with respect to the convex potential $\psi$ whose gradient is the proximal operator of $\phi$. This seems natural as approximating the gradient directly would yield an approximation of type (a) and thus not a lot of control on the Lipschitz continuity of the approximation. Moreover, in the learning community, the proximal operator is usually approximated by parametrizing a potential with a neural network and then taking the gradient of this potential. 
\begin{definition}{\textbf{Type (d) approximation.}} \label{def:approx_typed} We say that $\g_d = \nabla \psi_\epsilon$ is an $\epsilon$-type (d) approximation of $\prox_\phi$ if $\psi_\epsilon$ is continuously differentiable with Lipschitz continuous gradient, and is $\epsilon$-close to $\psi$ on a convex, compact $K \subset \Hi$:
\begin{equation}
    \sup_{x\in K} |\psi_\epsilon(x)-\psi(x)| \leq \epsilon.
\end{equation} 
\end{definition}
The convexity of $\psi_\epsilon$ is not required for this approximation to hold. However, relaxing the convexity of the approximant $\psi_\epsilon$ means that $\psi_\epsilon$ cannot be the proximal operator of some functional $\phi_\epsilon$ \cite{gribonval2020characterization}.

The first question we address in the construction of the approximation is the $L_\epsilon$-Lipschitz continuity of $\nabla \psi_\epsilon$. Usually, when $\psi_\epsilon$ is learned, practitioners enforce that $L_\epsilon <1$ (or at least try to) in order to derive convergence guarantees of algorithm using the contractivity of $\nabla \psi_\epsilon$ \cite{hurault2021gradient,hurault2022proximal,terris2021learning,bredies2024learning}. In the following we show that this can yield to sub-optimal approximation. In general one cannot obtain arbitrarily good approximation of Lipschitz continuous proximal operators, if the Lipschitz constant of the approximation is too small, i.e., if $L_\epsilon < L_\psi$. Before that we take a quick detour to discuss the approximation of Lipschitz continuous mappings.

There is a great body of literature on the approximation of Lipschitz continuous functions \cite{attouch1993approximation,lasry1986remark,bauschke2017,azagra2012real,cobzacs2019lipschitz}. Under several settings it was shown that approximants (i.e, for all $x$, $\Vert \psi - \psi_\epsilon(x) \Vert \leq \epsilon(x)$) exist such that $L_\psi \leq L_\epsilon$ \cite{azagra2012real,azagra2007smooth,johanis2024note}. On the other hand, the Lipschitz continuity of the gradient of the approximants was less studied. For convex functions, it is well known that the Moreau envelope provides a universal approximation technique of $\psi$. Recall that we note $u_\lambda$ the Moreau envelope of $\psi$ of parameter $\lambda$. For all $\psi \in \Gamma_0(\Hi)$ we have by \cite[Proposition 12.33]{bauschke2017} that for all $x \in \Hi$, $u_\lambda(x) \uparrow \psi(x)$ as $\lambda \downarrow 0$. Yet, the Lipschitz constant of the gradient of $u_{\lambda}$ is $1/\lambda$ \cite[Proposition 12.30]{bauschke2017}. Hence, the higher the Lipschitz constant of $\nabla u_\lambda$, the better the approximation. In general, it has been shown \cite{czarnecki2006approximation,azagra2013global} that convex $C^1$ functions can be approximated by convex $C^\infty$ functions in the $C^1$-fine topology, i.e., both the function and the gradient pointwise errors are controlled by a continuous function $\epsilon$ of $x$. Unfortunately, these results do not allow us to control the Lipschitz constant of the gradient of the approximant.

Hence, we show that for some convex functions $\psi$ belonging to $C^{1,1}(\Hi)$, the Lipschitz smoothness of the approximant $\psi_\epsilon$ is lower bounded by $L_\psi$, in a similar manner as if we directly tried to approximate the gradient mapping. In addition, we provide a series of results highlighting that without precise knowledge of the function $\psi$ we are trying to approximate with $\psi_\epsilon$, one should be careful about imposing constraints on the Lipschitz constant of $\nabla \psi_\epsilon$.

\begin{theorem} \label{th:Lipschitzness_lower_bound}
    Let $\psi:\RR^N \to \overline{\RR}$, be a $C^{1,1}$ convex function, whose gradient is locally $L_\psi$-Lipschitz continuous on $K$ a compact convex set of $\RR^N$ and define the measure space $(K,\mathcal{B}(K),\left.\lambda^N\right|_{K})$ where $\lambda^N$ is the $N-$dimensional Lebesgue measure and $\mathcal{B}(K)$ the Borel $\sigma-$algebra. Let $\psi_\epsilon:\Hi \to \RR$ be a locally $L_\epsilon$-smooth function on $K$ and $\epsilon>0$, such that on $K$ we have
    \begin{equation}
        \sup_{x\in K} |\psi_\epsilon(x) -\psi(x) | \leq \epsilon.
    \end{equation}
    Then for every $\eta>0$, there exists $t_0>0$ such that 
    \begin{equation}
        L_\psi - \eta < L_\epsilon + \frac{4 \epsilon}{t_0^2}.
    \end{equation}
    $t_0$ is bounded away from $0$ if the curvature is maximal away from the boundary of $K$, i.e., there exists $\delta>0$ such that for every $\eta>0$ 
    \begin{equation}
        E_\eta = \left\{ x \in \mathrm{int}~K \Big| \nabla^2 \psi(x) \text{ exists and } \Vert \nabla^2 \psi(x) \Vert_{op} >L_\psi -\eta\right\} \bigcap \left\{ x \in K \Big| \mathrm{dist}(x,\partial K)>\delta\right\}\neq \emptyset,
    \end{equation}
    where $\partial K$ denotes the boundary of $K$.
\end{theorem}
\begin{proof}
    Recall Alexandrov's theorem \cite{alexandrov1939almost,busemann1936krummungsindikatrizen}: if $\psi$ is convex, then the set-valued function $\nabla \psi$ is differentiable almost everywhere. Two facts arise from this. There exists a quadratic expansion of $\psi$ almost everywhere (\cite[Chapter 13]{VarAnalRockafellar},\cite{rockafellar1999second}), written as 
\begin{equation*}
    \psi(x+tw) = \psi(x) + t\langle \nabla \psi(x),w \rangle + \frac{t^2}{2}\langle\nabla^2 \psi(x) w, w \rangle + o(t^2 \Vert w \Vert^2).
\end{equation*}
In particular, for $t>0$
\begin{equation*}
    \frac{\psi(x+tw) + \psi(x-tw)-2\psi(x)}{t^2} \underset{t\to 0}{\to}  \langle \nabla^2 \psi(x) w, w \rangle.
\end{equation*}
Moreover $\psi \in C^{1,1}(K)$ with Lipschitz constant of the gradient $L_\psi$ then (by Rademacher's theorem and the definition of the derivative \cite[Section 3.1]{evans2018measure})
\begin{equation*}
    \mathrm{ess }\sup_{x\in K} \Vert \nabla^2 \psi(x)\Vert_{op} \leq L_\psi<+\infty
\end{equation*}
This inequality is in fact an equality. Indeed, denote by $M$ the essential supremum norm of $\nabla^2 \psi$. Take $x,y \in K$, and $f: t \to \nabla \psi(x + t(y-x))$ for $t\in[0,1]$. $f$ is differentiable almost everywhere in $[0,1]$ and we have $f'(t) = \nabla^2 \psi(x+t(y-x))(y-x)$. As $\nabla \psi$ is Lipschitz continuous, it is absolutely continuous on lines and hence,
\begin{equation*}
    f(1)-f(0) = \int_0^1 \nabla^2 \psi(x+t(y-x))(y-x) dt \leq \int_0^1 M \Vert y-x \Vert dt =  M \Vert y-x \Vert.
\end{equation*}
Thus, by definition of the Lipschitz continuity, $L_\psi \leq M$. Therefore $M= L_\psi$.

Now construct for all $k \in \mathbb{N}\backslash \{0\}$ the set
\begin{equation*}
    E_k = \left\{ x \in \mathrm{int}~K \Big| \nabla^2 \psi(x) \text{ exists and } \Vert \nabla^2 \psi(x) \Vert_{op} >L_\psi -\frac{1}{k}\right\}
\end{equation*}
For each $k$, by definition of the essential supremum norm, $E_k$ is measurable. Now construct a sequence $\{x_k\}_{k\in\mathbb{N}\backslash\{0\}}\in \mathrm{int}~K$, where $\psi$ is twice differentiable and such that 
\begin{equation*}
   \Vert \nabla^2 \psi(x_k) \Vert_{op} \uparrow_{k\to+\infty}  L_\psi. 
\end{equation*}
Such sequence exists by Alexandrov's theorem \cite{alexandrov1939almost}. Where it exists, $\nabla^2 \psi(x)$ is symmetric, hence $\Vert \nabla^2 \psi(x) \Vert_{op} = \lambda_{\max}(\nabla^2 \psi(x))$. Construct a sequence of unit eigenvectors $\{u_k\}_{k \in \mathbb{N}\backslash\{0\}}$ of the eigenvalues $\lambda_{\max}(\nabla^2 \psi(x_k))$. Let $t>0$. We have
\begin{equation*}
    \frac{\psi(x_k+tu_k) + \psi(x_k-tu_k)-2\psi(x_k)}{t^2} = \lambda_{\max}(\nabla^2 \psi(x_k)) + o(1).
\end{equation*}
For every $\eta$, there exists $k_0 \in \mathbb{N}\backslash\{0\}$, such that for all $k>k_0$ there exists $t_k$>0 such that $[x_k -t_k u_k, x_k + t_k u_k]\in K$. We have
\begin{equation*}
    \frac{\psi(x_k+t_k u_k) + \psi(x_k-t_k u_k)-2\psi(x_k)}{t_k^2} > L_\psi - \eta.
\end{equation*}
On the other hand, if $\sup_{x \in K} |\psi_\epsilon(x) - \psi(x)|=\epsilon$, then
\begin{align*}
    \frac{\psi(x_k+t_k u_k) + \psi(x_k-t_k u_k)-2\psi(x_k)}{t_k^2} & \leq \frac{\psi_\epsilon(x_k+t_k u_k) + \psi_\epsilon(x_k-t_k u_k)-2\psi_\epsilon(x_k)}{t_k ^2} + \frac{4\epsilon}{t_k ^2} \\& \leq L_\epsilon \Vert u_k \Vert^2 +\frac{4\epsilon}{t_k^2} = L_\epsilon + \frac{4\epsilon}{t_k^2}
\end{align*}
Now, taking $t_0 = \inf_{k\geq k_0} t_k$, we get 
\begin{equation*}
    L_\psi - \eta < L_\epsilon + \frac{4 \epsilon}{t_0^2}.
\end{equation*}
If the maximum of the curvature is "attained away" from the boundary, then $t_0$ is bounded away from $0$. Indeed, in this case there exists $\delta>0$ such that the following set is not empty for all $\eta>0$
\begin{equation*}
        E_\eta = \left\{ x \in \mathrm{int}~K \Big| \nabla^2 \psi(x) \text{ exists and } \Vert \nabla^2 \psi(x) \Vert_{op} >L_\psi -\eta\right\} \bigcap \left\{ x \in K \Big| \mathrm{dist}(x,\partial K)>\delta\right\},
    \end{equation*}
    then $\inf_{k\geq k_0} t_k \geq \delta$.
\end{proof}
\begin{remark}
    Convexity of $\psi$ is necessary for Theorem \ref{th:Lipschitzness_lower_bound} to hold. $h:x \mapsto x$ approximates with error $\epsilon$ any $\psi:x \mapsto x + \epsilon \cos(\frac{x}{\epsilon})$ even though $h''(x) = 0$, while $\Vert\psi''(x)\Vert_{\ell_\infty}=1/\epsilon$. 
\end{remark}
\begin{remark}
    The maximum curvature is attained away from the boundary for instance for the functions $\psi$ associated with the proximal operator of the $\ell_1$-norm, the $\ell_2$-norm or the squared $\ell_2$-norm.
\end{remark}
If this theorem highlights an impossibility to approximate all $C^{1,1}$ convex functions with arbitrary precisions when $L_\epsilon<L_\psi$, we can also see that in some particular cases, one can reach such arbitrary precision while allowing smaller $L_\epsilon$. Indeed, if the maximum value of the curvature is attained at the boundary then $t_0$ goes to $0$ (which renders the bound essentially vacuous). We can illustrate this fact in dimension 1. Take $\psi:[0,1]\to \RR, x\mapsto \frac{x^2}{2}-\frac{x^3}{12}$. Then, $\psi''(x) = 1-\frac{x}{2}$, attaining its maximum value at $x=0$. By construction, the sets $E_k$ are 
\begin{equation*}
    \left\{ x \in [0,1] \Big| \psi''(x)>1 -\frac{1}{k}\right\}.
\end{equation*}
Therefore, one can see that $E_k \to \{0\}$ as $k$ goes to infinity, therefore $t_0\downarrow 0$. We can construct $\psi_n$ an approximation of $\psi$ such that $L_{\psi_n}<L_\psi$. Take for all $n\geq 1$, $\psi_n(x):= \psi(x) - \frac{x^2}{2n}$. $\psi_n$ is $\frac{1}{2n}$-close to $\psi$ on $[0,1]$.

On the other hand, we can construct the following example on any segment $[-a,a]:=\{\lambda (-a) + (1-\lambda) a| \lambda \in [0,1] \}$, where $a\in\Hi$ and $\Vert a \Vert = \epsilon$. 
\begin{proposition} \label{prop:quality_typed}Let $L>0$, and $\psi: x \to \frac{L}{2}\Vert x \Vert^2$. We approximate $\psi$ with $\psi_\epsilon$, a twice continuously differentiable function with $L'$-Lipschitz continuous gradient. Suppose that $L'<L$. On any segment $[-a,a]$ where $a\in\Hi$ and $\Vert a \Vert = \epsilon$, we have
\begin{equation}
    \sup_{x\in[-a,a]} |\psi_\epsilon(x)-\psi(x)| \geq (L-L') \frac{\epsilon^2}{4}.
\end{equation}
\end{proposition}
\begin{proof}
    We have $\psi_\epsilon''(x) - \psi''(x) = \psi_\epsilon''(x) -L \leq L'-L <0$. Thus $\psi-\psi_\epsilon$ is $L-L'$-strongly convex. Hence, by \cite[Definition 5.16]{beck2017first}, we have for $\lambda \in [0,1]$
    \begin{equation*}
        -(\psi_\epsilon-\psi)(-\lambda a + (1-\lambda) a) \leq -\lambda (\psi_\epsilon-\psi)(-a) - (1-\lambda)(\psi_\epsilon- \psi)(a) - \frac{L-L'}{2}\lambda(1-\lambda) \Vert 2 a \Vert^2.
    \end{equation*}
    Set $\lambda = 1/2$. We have
    \begin{align*}
        -(\psi_\epsilon-\psi)(0) & \leq -\frac{1}{2}\left[(\psi_\epsilon-\psi)(-a) + (\psi_\epsilon-\psi)(a) + (L-L') \Vert a \Vert^2 \right].
    \end{align*}
    Denote by $M:=\sup_{x \in [-a,a]} |(\psi_\epsilon-\psi)(x)|$. Then,
    \begin{equation*}
        M \geq \frac{1}{2}\left[-M - M + (L-L') \Vert a \Vert^2 \right],
    \end{equation*}
    hence,
    \begin{equation*}
        M \geq \frac{(L-L') \Vert a \Vert^2}{4}.
    \end{equation*}
\end{proof}
We end this discussion by emphasizing that it does not seem that there exist rules to obtain the location of maximum curvature of $\psi$, without knowing $\psi$ explicitly. 
\paragraph{Quality of the approximation.}
We know now that the location of the points of highest curvature has a direct impact on the values of $L_\epsilon$ with respect to the values of $L_\psi$, a similar observation can be made for the quality of the approximation of the proximal operator with type (d) approximation, i.e., how $\Vert \prox_\phi(x) - \g_d(x) \Vert$ behaves with respect to the location of $x$ in $K$.
\begin{proposition} \label{prop:sigma_d}
    Let $\psi:\Hi \to \overline{\RR}$ be a continuously differentiable convex function with $L_\psi$-Lipschitz continuous gradient. Let $K$ be a compact convex subset of $\Hi$ and let $\epsilon>0$ such that the continuously differentiable function $\psi_\epsilon:\Hi \to \overline{\RR}$ with 
    $L_\epsilon$ Lipschitz continuous gradient, satisfies
    \begin{equation}
        \Vert \psi_\epsilon - \psi \Vert_{\ell_\infty(K)}= \epsilon.
    \end{equation}
    Then, 
    \begin{equation*}
    r(x)= \min\left\{\frac{1}{L_\epsilon}, \sup\{t>0, x - t(\nabla \psi_\epsilon(x) - \nabla \psi(x)) \in K\}\right\}.
    \end{equation*} 
    For all $0<t\leq r(x)$
     \begin{equation}
        \Vert \nabla \psi_\epsilon(x) - \nabla \psi(x) \Vert_2 \leq 2\sqrt{\frac{\epsilon}{t}} \frac{1}{\sqrt{2-L_\epsilon}}.
     \end{equation}
     In particular when $r(x) = \frac{1}{L_\epsilon}$ we have,
     \begin{equation}
        \Vert \nabla \psi_\epsilon(x) - \nabla \psi(x) \Vert_2 \leq 2\sqrt{L_\epsilon\epsilon}.
    \end{equation}
\end{proposition}
\begin{proof}
    By the descent lemma, and convexity of $\psi$, we have for all $x,y \in K$ that
    \begin{equation*}
        \psi(y)- \psi(x) +\psi_\epsilon(y) - \psi_\epsilon(x) + \langle \nabla \psi(x), y-x \rangle + \langle \nabla \psi_\epsilon(x), x-y \rangle \leq \frac{L_\epsilon}{2}\Vert y-x \Vert^2_2 
    \end{equation*}
    which yields by assumption
    \begin{equation*}
        \langle \nabla \psi_\epsilon(x)- \nabla \psi(x), x-y \rangle \leq \frac{L_\epsilon}{2}\Vert y-x \Vert^2_2 +2 \epsilon
    \end{equation*}
    Now, take $y = x - t(\nabla \psi_\epsilon(x) - \nabla \psi(x))$ with $t>0$ and such that $y\in K$. We have 
    \begin{equation*}
        \Vert \nabla \psi_\epsilon(x) - \nabla \psi(x) \Vert^2_2 \leq \frac{tL_\epsilon }{2} \Vert \nabla \psi_\epsilon(x) - \nabla \psi(x) \Vert^2_2 + \frac{2\epsilon}{t},
    \end{equation*}
    rearraging,
    \begin{equation*}
        \Vert \nabla \psi_\epsilon(x) - \nabla \psi(x) \Vert^2_2 \leq \frac{4\epsilon}{t} \frac{1}{2-tL_\epsilon}
    \end{equation*}
    The right hand side is minimal when $t\mapsto t(2-tL_\epsilon)$ is maximal, i.e., when $t = \frac{1}{L_\epsilon}$. It yields
    \begin{equation*}
        \Vert \nabla \psi_\epsilon(x) - \nabla \psi(x) \Vert_2 \leq 2\sqrt{L_\epsilon\epsilon}.
    \end{equation*}
    This can only hold whenever $K$ is large enough to contain for all $x\in K$ the point $x-\frac{1}{L_\epsilon} (\nabla \psi_\epsilon(x) - \nabla \psi(x))$. Otherwise, we have the position-aware bound. Define 
    \begin{equation*}
    r(x)= \min\left\{\frac{1}{L_\epsilon}, \sup\{t>0, x - t(\nabla \psi_\epsilon(x) - \nabla \psi(x)) \in K\}\right\}.
    \end{equation*} 
    For all $0<t\leq r(x)$
     \begin{equation*}
        \Vert \nabla \psi_\epsilon(x) - \nabla \psi(x) \Vert_2 \leq 2\sqrt{\frac{\epsilon}{t}} \frac{1}{\sqrt{2-tL_\epsilon}}
     \end{equation*}
\end{proof}
\begin{remark}
    To allow a tighter Lipschitz continuity of $\nabla \psi_\epsilon$, we can reduce the diameter of $K$ (thus reducing the maximum value of $t_0$), but if $t_0$ is bounded away from $0$, we inevitably pay a cost in the quality of the intended approximation: $L_\epsilon$ is not independent of $\epsilon$ through Theorem \ref{th:Lipschitzness_lower_bound}. This dependence also encodes the dependence w.r.t. $L_\psi$ of this bound.
\end{remark}
This result could be generalized non-convex $\psi$, and would make appear explictly the Lipschitz constant $\psi$. As an example, such a bound can be deduced from the Landau–Kolmogorov inequality, which upper bounds the infinite norm of the derivatives of 1-dimensional functions on intervals.
\begin{proposition}
    Let $\psi:\RR \to \overline{\RR}$ be a twice continuously differentiable function with $L_\psi$-Lipschitz continuous derivative. Let $I \subset \RR$ and let $\epsilon>0$ such that a twice continuously differentiable function $\psi_\epsilon:\RR \to \overline{\RR}$ with 
    $L_\epsilon$ Lipschitz continuous derivative, satisfies
    \begin{equation}
        \Vert \psi_\epsilon - \psi \Vert_{\ell_\infty(I)}= \epsilon.
    \end{equation}
    Then,
    \begin{equation}
        \Vert \psi_\epsilon'- \psi' \Vert_{\ell_\infty(I)} \leq 2 \sqrt{(L_\psi + L_\epsilon)\epsilon }.
    \end{equation}
\end{proposition}
\begin{proof}
    We have by the Landau-Kolmogorov inequality (or in this particular case the Kallman-Rota inequality \cite{kallman1970inequality}) that
    \begin{align*}
        \Vert \psi_\epsilon'-\psi' \Vert_{\ell_\infty(I)}^2 &\leq 4 \Vert \psi_\epsilon - \psi \Vert_{\ell_\infty(I)} \Vert \psi_\epsilon''-\psi'' \Vert_{\ell_\infty(I)} \\
        & \leq 4(L_\psi+L_\epsilon) \epsilon
    \end{align*}
    Hence,
    \begin{equation*}
        \Vert \psi_\epsilon'-\psi' \Vert_{\ell_\infty(I)} \leq 2\sqrt{ (L_\psi+L_\epsilon)\epsilon}.
    \end{equation*}
\end{proof}
\paragraph{Admissibility of $\g_d$.} For this approximation, we cannot recover the lim sup inclusion like for the previous approximations, without assuming some continuity of the approximations $\psi_\epsilon$ with respect to $\epsilon$. This requires a new definition of type (d)-approximations which would certainly dilute the generality of the one we have here.
\subsection{Approximations of type (e)}
The approximation of type (e) is constructed by enlarging the subdifferential of the convex potential $\psi$, thus generalizing \cite[Theorem 1]{gribonval2020characterization}.
\begin{definition}{\textbf{Type (e) approximation.}} \label{def:approx_typee}
    We say that $\g_e$ is an $\epsilon$-type (e) approximation of $\prox_\phi$ if for all $y \in \Hi$
    \begin{equation}
        \g_e(y) \in \partial_\epsilon \psi(y).
    \end{equation}
\end{definition}
Assuming that $\psi$ is $L$-Lipschitz smooth, we have that:
\begin{lemma} \label{lm:espilon_subdiff_smooth_bound}
    Let $\psi$ be a proper, l.s.c., continuously differentiable convex function, with $L_\psi$-Lipschitz continuous gradient. The $\epsilon$-subdifferential of $\psi$ at $x \in \dom \psi$ is such that
    \begin{equation}
        \partial_\epsilon \psi(x) \subseteq \{ s \mid \Vert s - \nabla \psi(x) \Vert \leq \sqrt{2 L_\psi \epsilon} \}.
    \end{equation}
\end{lemma}
\begin{proof}
    We have $s \in \partial_\epsilon \psi(x)$ equivalent to
    \begin{equation*}
        (\forall y \in \dom \psi), \quad \psi(x) + \langle s , y-x \rangle - \epsilon \leq \psi(y).
    \end{equation*}
    And the $L_\psi$-smoothness of $\psi$ is equivalent to 
    \begin{equation*}
        (\forall y \in \dom \psi), \quad \psi(x) + \langle \nabla \psi(x), y-x \rangle + \frac{L_\psi}{2} \Vert y - x \Vert^2 \geq \psi(y)
    \end{equation*}
    Hence we have for all $y \in \dom \psi$
    \begin{align*}
        \langle \nabla \psi(x)-s, y-x \rangle + \frac{L_\psi}{2} \Vert y - x \Vert^2 & \geq  - \epsilon. \\
       -\langle \nabla \psi(x)-s, y-x \rangle - \frac{L_\psi}{2} \Vert y - x \Vert^2-\epsilon & \leq 0
    \end{align*}
    The supremum of the left-hand-side is attained at $y = x - \frac{1}{L_\psi} \left(\nabla \psi(x) - s \right)$. Thus,
    \begin{equation*}
        \frac{1}{L_\psi} \langle \nabla \psi(x), s - \nabla \psi(x) \rangle + \frac{1}{2L_\psi} \Vert s - \nabla \psi(x) \Vert^2 \geq \frac{1}{L_\psi}\langle s , s-\nabla \psi(x) \rangle - \epsilon,
    \end{equation*}
    which is equivalent to
    \begin{equation*}
        \left( \frac{1}{2L_\psi}- \frac{1}{L_\psi} \right) \Vert s- \nabla \psi(x) \Vert^2 \geq - \epsilon,
    \end{equation*}
    and thus
    \begin{equation*}
        \Vert s- \nabla \psi(x) \Vert \leq \sqrt{2L_\psi\epsilon}.
    \end{equation*}
\end{proof}
From this inclusion we extract a necessary condition for $x$ to be a fixed point of $\g_e$
\begin{equation}
    x \in \Fix \g_e \implies x \in \partial_\epsilon \psi(x) \implies \Vert x- \nabla \psi(x) \Vert \leq \sqrt{2L_\psi\epsilon}
\end{equation}
In general we have
\begin{proposition} \label{prop:sigma_e}
    Suppose that Assumption \ref{ass:weaklyconvex} holds. Let $\g_e$ be an $\epsilon$-type (e) approximation of $\prox_\phi$ (which is $L_\psi$-Lipschitz continuous). Then, we have
    \begin{equation}
        (\forall x \in \dom \psi), \quad \Vert \g_e(x) - \prox_\phi(x) \Vert_2 \leq \sqrt{2 L_\psi \epsilon}.
    \end{equation}
    In particular, if fixed point exist,
    \begin{equation}
       (\forall \bar x \in \Fix \g_e), \quad  \Vert \bar x - \prox_\phi(\bar x) \Vert_2 \leq \sqrt{2L_\psi\epsilon}.
    \end{equation}
\end{proposition}

\paragraph{Lipschitz continuity of $\g_e$.}
First, we recall some classical results.
The Hausdorff distance between two (nonempty) compact convex sets A and B
is defined as follows \cite[Theorem V.3.3.8]{hiriart2013convex}
\begin{equation}
  \Delta_H(A,B)
  = \max\bigl\{\, \lvert \sigma_A(d) - \sigma_B(d) \rvert : \Vert d \Vert=1 \,\bigr\}.
\end{equation}
where $\sigma_A$ is the support function of $A$ \cite[Definition V.2.1.1]{hiriart2013convex}.
\begin{definition}
Let $A$ be a nonempty set in $\RR^n$. The function $\sigma_A : \RR^n \to \RR \cup \{+\infty\}$
defined by
\begin{equation}
  \RR^n \ni x \longmapsto \sigma_A(x)
  := \sup\{ \langle a, x \rangle : a \in A \}
\end{equation}
is called the \emph{support function} of $A$.
\end{definition}
\begin{theorem}\cite[Theorem XI.4.1.3]{hiriart-urruty1993convexII}
Let $\psi:\RR^n \to \RR$ be a convex Lipschitzian function on $\RR^n$.
Then there exists $K>0$ such that, for all $x,x'\in\RR^n$ and
$\epsilon,\epsilon'>0$,
\begin{equation}
\Delta_H\!\bigl(\partial_{\epsilon} \psi(x),\, \partial_{\epsilon'} \psi(x')\bigr)
  \le \frac{K}{\min\{\epsilon,\epsilon'\}}
     \bigl(\Vert x-x'\Vert + \lvert \epsilon - \epsilon' \rvert \bigr).
\end{equation}
\end{theorem}
Assuming $\psi$ is $L$-smooth, then we have the following bound.
\begin{proposition} \label{prop:lipschitz_e}
    Let $\psi:\RR^n \to \RR$ be a proper, l.s.c., continuously differentiable convex function with $L_\psi$-Lipschitz gradient. Then, for $\epsilon>0$, for all $x,x'\in\RR^n$
    \begin{equation}
        \Delta_H(\partial_\epsilon \psi(x), \partial_\epsilon \psi(x')) \leq L_\psi \Vert x - x'\Vert.
    \end{equation}
\end{proposition}
\begin{proof}
    Recall that by Lemma \ref{lm:espilon_subdiff_smooth_bound}, we have
    \begin{equation*}
        (\forall x \in \RR^n), \quad \partial_\epsilon \psi(x) \subseteq \mathbb{B}(\nabla \psi(x), \sqrt{2L_\psi\epsilon}).
    \end{equation*}
    Yet, the Hausdorff distance between two balls of identical radius in a normed space is given by
    \begin{equation*}
        (\forall x,x' \in \RR^n), \quad \Delta_H(\mathbb{B}(\nabla \psi(x), \sqrt{2L_\psi\epsilon}), \mathbb{B}(\nabla \psi(x'),\sqrt{2L_\psi\epsilon})) = \Vert \nabla \psi(x) - \nabla \psi(x') \Vert. 
    \end{equation*}
    Hence,
    \begin{equation*}
        (\forall x,x' \in \RR^n), \quad \Delta_H(\partial_\epsilon \psi(x), \partial_\epsilon \psi(y)) \leq L_\psi \Vert x - x' \Vert + \sqrt{2L\psi\epsilon}.
    \end{equation*}
    The $\sqrt{2L\psi\epsilon}$ term accounts for the fact that $\partial_\epsilon \psi(x)$ (and $\partial_\epsilon \psi(y)$) may not be centered around their respective balls.
\end{proof}
The Hausdorff distance is a worst case scenario, hence we can expect that with more assumptions on $\psi$ we could derive a sharper Lipschitz bound.

On the other hand, if $\psi$ is only Lipschitz continuous (not necessarily differentiable), we can still construct a Lipschitz continuous approximation of $\prox_\phi$ with some convex geometry and the notion of Steiner point \cite{fefferman2017sharp,fefferman2018sharp,shvartsman2004barycentric,steiner1881gesammelte}.
\begin{proposition} \label{prop:steiner_point}
    Suppose that $\psi:\RR^n\to\RR$ is a convex Lipschitzian function on $\RR^n
    $. Let $\partial_\epsilon \psi(\RR^n) = \{ \partial_\epsilon \psi(x) \mid \forall x \in \RR^n \}$. Then for all $x \in \mathrm{int} ~\dom \psi$, there exists a Lipschitz continuous selector $s:\partial_\epsilon \psi(\RR^n) \to \RR^n$, the Steiner point. For $K$, a compact convex subset of $\RR^n$, the Steiner map is equal to
    \begin{equation}
        s(K) = n \int_{\mathbb{S}_{n}} u\, \sigma_K(u)\, d\mu(u).
    \end{equation}
Here $\mathbb{S}_n$ is the unit sphere in $\RR^n$, $\sigma_K$ the support function of $K$, and $\mu$ denotes the normalized Lebesgue measure on $\mathbb{S}_n$ (in the canonical basis for instance).

This selector verifies for all $x\in\RR^n$, $s(\partial_\epsilon \psi(x)) \in \partial_\epsilon \psi(x)$ and,
    \begin{equation}
        (\forall x \in \RR^n), (\forall x' \in \RR^n), \quad \Vert s(\partial_\epsilon \psi(x)) - s(\partial_\epsilon\psi(x')) \Vert \leq c_n \frac{K}{\epsilon} \Vert x - x' \Vert,
    \end{equation}
    where $c_n =  2\pi^{-\tfrac{1}{2}}\Gamma\!\left(\tfrac{n}{2} + 1\right)/\Gamma\!\left(\tfrac{n+1}{2}\right)\sim \sqrt{n}$, and $K>0$ comes from \cite[Theorem XI.4.1.3]{hiriart-urruty1993convexII}.
\end{proposition}
\begin{proof}
    By \cite[Theorem XI.1.1.4]{hiriart-urruty1993convexII}, for all $\epsilon \geq 0$, for all $x \in \mathrm{int}~\dom \psi$, $\partial_\epsilon \psi(x)$ is a bounded closed convex set. The Steiner map is Lipschitz continuous \cite{shvartsman2004barycentric} in the Hausdorff metric with Lipschitz constant $c_n$. Hence,
    \begin{equation*}
        (\forall x \in \RR^n), (\forall x' \in \RR^n), \quad \Vert s(\partial_\epsilon \psi(x)) - s(\partial_\epsilon\psi(x')) \Vert \leq c_n \Delta_H(\partial_\epsilon \psi(x),\partial_\epsilon \psi(x')).
    \end{equation*}
    We conclude by invoking \cite[Theorem XI.4.1.3]{hiriart-urruty1993convexII}.
\end{proof}
If $\psi$ is L-smooth, the Steiner point map coincides with the gradient of $\psi$: invoke additivity with respect to Minkowski addition to write:
\begin{equation}
   (\forall x \in \RR^n), \quad s(\partial_\epsilon \psi(x)) = s(\{\nabla \psi(x) \}) + s\left(\mathbb{B}(0,\sqrt{2L_\psi\epsilon})\right) = \nabla \psi(x) + 0 = \nabla \psi(x).
\end{equation}
Unfortunately, no upper bound depending on $\epsilon$ exists, in general, on the distance between $\prox_\phi$ and $s(\partial_\epsilon \psi(\cdot))$, as a consequence of Br{\o}ndsted--Rockafellar theorem \cite[Theorem XI.4.2.1]{hiriart-urruty1993convexII}.
\paragraph{Admissibility of $\g_e$.} As the $\epsilon$-subdifferential grows around $\nabla \psi$, we can construct a sequence of $\g_e^\epsilon$ such that  
\begin{equation}
    \bigcap_{\epsilon \geq 0} \Fix \g^\epsilon_e = \Fix \prox_\phi.
\end{equation}
Such sequence can be constructed with the Steiner point argument. However, in constrast with approximations of type (c) it is not straightforward to conclude that every approximations of type (e) has fixed points. 
\subsection{Approximation of type (f)} This last approximation is more suited to zeroth-order optimization, since it can be computed only using evaluations of $\phi$. As a result the quality of the approximation is significantly worse than other approximation methods, but it is the price to pay to do zeroth order optimization \cite{liu2020primer}.
\begin{definition}{\textbf{Type (f) approximation.}} \label{def:approx_typef}
    We say that $\g_f$ is an $\epsilon$-type (f) approximation of $\prox_\phi$ if for all $y \in \Hi$
    \begin{equation}
        \g_f(y) = y - \lambda \nabla u^\epsilon(y).
    \end{equation}
    where $\lambda>0$, $u^\epsilon$ is $\sqrt{\epsilon}$ close to $u$, the Moreau envelope of $\phi$ of parameter $\lambda$ as the solution of \begin{equation}
    \left\{ \begin{array}{cc}
    \frac{\partial }{\partial \lambda}u^\epsilon + \frac{1}{2} \Vert \nabla_x u^\epsilon \Vert^2 = \frac{\epsilon}{2}\Delta_x u^\epsilon & \text{in } \Hi \times (0,\Lambda] \\
            u^\epsilon(x,0) = \phi(x) & \text{on } \Hi \times \{0\}.
    \end{array}\right. \label{eq:HJ_viscous}
\end{equation}
\end{definition}
If we do not construct $u^\epsilon$ as the solution of this equation \eqref{eq:HJ_viscous} and only enforce $\epsilon$-closeness to $u$, then we fall under type (d) approximation.%

This construction was proposed in \cite{osher2023hamilton}, where the authors exploit the fact that the Moreau envelope is a viscosity solution to a Hamilton-Jacobi equation in order to derive an approximation method of the proximal operator. For completeness of the presentation, we prove this statement here using only elements of convex analysis. %
\begin{proposition}
    Let $\phi\in\Gamma_0(\Hi)$. Let $\lambda>0$ and note $u$ the Moreau envelope of $\phi$, $\prox_{\lambda \phi}$ its proximal operator, defined by
    \begin{equation}
        (\forall x \in \Hi), \quad \prox_{\lambda \phi}(x) \overset{\triangle}{=} \argmin_{z \in \Hi} \Phi(x,\lambda,z) = \phi(z) + \frac{1}{2\lambda} \Vert z-x \Vert^2, \quad u(x,\lambda) \overset{\triangle}{=} \min_{z \in \Hi} \Phi(x,\lambda,z).
    \end{equation}
    Then, let $\Lambda>0$, for all $\lambda\in [0,\Lambda]$, $u$ is a viscosity solution to the Hamilton-Jacobi equation (Burgers' equation)
    \begin{equation}
        \left\{ \begin{array}{cc}
            \frac{\partial}{\partial \lambda}u(x,\lambda) + \frac{1}{2} \Vert \nabla_x u(x,\lambda) \Vert^2 = 0 & \text{in } \Hi \times (0,\Lambda] \\
            u(x,0) = \phi(x) & \text{on } \Hi \times \{0\} 
        \end{array}\right.
    \end{equation}
\end{proposition}
\begin{proof}
    Denote by $u_x(\lambda)\overset{\triangle}{=} u(x,\lambda)$ and $\Phi_x \overset{\triangle}{=} \Phi(x,\lambda,z)$ for a fixed $x$. Set $\lambda>0$. The Moreau envelope is everywhere differentiable with respect to $t$ and by Danskin's theorem \cite[Section 6.7]{bertsekas2009convex}, we have
    \begin{equation*}
        \frac{\partial}{\partial \lambda}(u_x)(\lambda) = \frac{\partial}{\partial \lambda} (\Phi_x)(\lambda,\prox_{\lambda \phi}(x)).
    \end{equation*}
    Hence,
    \begin{equation*}
        \frac{\partial}{\partial \lambda}(u_x)(\lambda) = \frac{-1}{2 \lambda^2} \Vert \prox_{\lambda \phi}(x) - x \Vert^2.
    \end{equation*}
    By \cite[Proposition 12.30]{bauschke2017}, we have
    \begin{equation*}
        \nabla_x u(x,\lambda) = \frac{1}{\lambda}(x - \prox_{\lambda \phi}(x)).
    \end{equation*}
    We have thus for all $\lambda>0$, $x\in \Hi$ 
    \begin{equation*}
        \frac{\partial }{\partial \lambda}u(x,\lambda) + \frac{1}{2} \Vert \nabla_x u(x,\lambda) \Vert^2 = 0.
    \end{equation*}
    We conclude in $\lambda=0$ by \cite[Proposition 12.33]{bauschke2017}.
\end{proof}
Solutions to this equation exist as soon as $\phi$ is continuous \cite{alvarez1999hopf}.
The associated viscous Hamilton-Jacobi equation for $\epsilon>0$ is equation \eqref{eq:HJ_viscous}. The solution to \eqref{eq:HJ_viscous} $u^\epsilon$ goes to $u$ uniformly as $\epsilon$ goes to $0^+$ \cite[Theorem 5.1]{crandall1984two}:  if $0 < \epsilon,\Lambda <+\infty$ and if $\phi$ is bounded and Lipschitz continuous (hence the restriction to compact convex sets $K$), there exists $C>0$ such that
\begin{equation*}
    \sup_{\lambda\in [0,\Lambda]} \sup_{x\in K} | u(x,\lambda) - u^\epsilon(x,\lambda) | \leq C \sqrt{\epsilon}.
\end{equation*}
Then, the authors of \cite{osher2023hamilton} show using the Cole-Hopf change of variable \cite[Section 4.4.1]{evans2022partial}, that 
\begin{equation}
    \nabla_x u^\epsilon(x,\lambda) = \frac{1}{\lambda}\left( x - \frac{\mathbb{E}_{y \sim \mathcal{N}(x,\epsilon \lambda)}[y \exp(-\phi(y)/\epsilon)]}{\mathbb{E}_{y \sim \mathcal{N}(x,\epsilon \lambda)}[\exp(-\phi(y)/\epsilon)]}\right)
\end{equation}
which yields a zeroth order approximation of $\prox_{\lambda \phi}$ as \cite[Theorem 1]{osher2023hamilton}
\begin{equation*}
    \prox_{\lambda \phi}(x) = \lim_{\epsilon \to 0^+} \frac{\mathbb{E}_{y \sim \mathcal{N}(x,\epsilon \lambda)}[y \exp(-\phi(y)/\epsilon)]}{\mathbb{E}_{y \sim \mathcal{N}(x,\epsilon \lambda)}[\exp(-\phi(y)/\epsilon)]}.
\end{equation*}
The expectations are approximated using Monte-Carlo sampling. A similar derivation can be found in \cite{tibshirani2025laplace}, which generalizes this approach to infimal convolutions with other functions than $\frac{1}{2} \Vert \cdot \Vert^2$. 

\paragraph{Quality of the approximation.} In \cite{di2025monte}, the authors show a convergence rate to the true value of the proximal operator in the convex case \cite[Theorem 2]{di2025monte}\footnote{In the version of the paper we read, the authors write that $\phi$ is only l.s.c., but their proof clearly invokes the convexity of $\phi$.}. Two simple modifications of their argument yield an improved convergence rate in the convex case (by a factor $\sqrt{2}$) and a convergence rate in the weakly convex case:
\begin{proposition} \label{prop:sigma_f}
    Let $\phi:\Hi \to \RR$ be a proper, l.s.c., $\rho$-weakly convex such that $\lambda^{-1}-\rho>0$. Let $\g_f$ be an $\epsilon$-approximation of type $f$ of $\prox_\phi$. Then,
\begin{equation}
    \sup_{x \in \Hi} \Vert \prox_{\lambda \phi}(x) - \g_f(x)\Vert \leq \sqrt{N (\lambda^{-1} - \rho)^{-1} \epsilon}
\end{equation}
\end{proposition}
\begin{proof}
    Denote $z^*(x) = \prox_{\lambda \phi}(x)$. From \cite[Appendix A]{di2025monte}, the approximation error can be bounded by
    \begin{equation*}
        \Vert z^*(x) - \g_f(x)\Vert \leq \sqrt{\mathbb{E}_{p_\epsilon}(\Vert W\Vert^2)},
    \end{equation*}
    where $W$ is random variable following the density 
    \begin{equation*} 
        p_\epsilon(w) \varpropto \exp\left(-\frac{f(w)}{\epsilon}\right) = \exp\left(-\frac{\Phi_{\lambda}(z^*(x) + w)+\Phi_\lambda(z^*(x))}{\epsilon}\right).
    \end{equation*}
    Now, as $\phi$ is $\rho-$weakly convex, $f$ is $(\lambda^{-1} - \rho)$-strongly convex, therefore $\partial f$ is $(\lambda^{-1} - \rho)$-strongly monotone \cite[Example 22.4(iv)]{bauschke2017}, i.e.,
    \begin{equation}
    \left(\forall (w,s_w) \in \mathrm{gra}\,\partial f\right) \; \left(\forall (v,s_v) \in \mathrm{gra}\,\partial f \right), 
\quad \langle w - v, s_w -  s_v \rangle \geq (\lambda^{-1}-\rho) \Vert w-v \Vert^2.
\end{equation}
If $f$ vanishes at $0$, and is minimized at $0$, then 
\begin{equation*}
    \forall w \in \Hi, \forall s_w \in \partial f(w), \quad \langle w,s_w \rangle \geq (\lambda^{-1}-\rho) \Vert w \Vert^2.
\end{equation*} 
$f$ is convex, therefore almost everywhere differentiable, yielding that $s_w$ is almost everywhere unique and equal to $\nabla f(w)$, thus
\begin{equation*}
    \mathbb{E}_{p_\epsilon}(\Vert W \Vert^2) \leq (\lambda^{-1}-\rho)^{-1} \mathbb{E}_{p_\epsilon}(\langle W, \nabla f (W) \rangle).
\end{equation*}
The rest of the computations are identical to those in \cite[Appendix A]{di2025monte}, and we have 
\begin{equation*}
    \mathbb{E}_{p_\epsilon}(\langle W, \nabla f (W) \rangle) = N \epsilon,
\end{equation*}
which concludes the proof.
\end{proof}
\paragraph{Lipschitz continuity of $\g_f$.} Moreover, we can obtain an upper bound on the Lipschitz constant of $\nabla_x u^\epsilon (\cdot,\lambda)$ with the Cole-Hopf change of variable $f= \exp(-\frac{u^\epsilon}{\epsilon})$.
\begin{proposition} \label{prop:lipschitz_f}
    Let $\epsilon>0$ and let $u^\epsilon$ be a solution to the viscous HJ equation \eqref{eq:HJ_viscous}. Assume now that $\inf u^\epsilon>-\infty$. Then for all $\lambda>0$,
    \begin{equation}
    (\forall x \in \Hi), \quad D^2_x(u^\epsilon)(x,\lambda) \preceq  \frac{1}{\lambda}.
    \end{equation}
\end{proposition}
\begin{proof}
    The Cole-Hopf change of variable $f= \exp(-\frac{u^\epsilon}{\epsilon})$ solves the heat equation on $(0,T]$
    \begin{equation*}
    \frac{\partial}{\partial \lambda} f(x,\lambda) = \frac{\epsilon}{2} \Delta_x f(x,\lambda) \quad \text{in } \Hi \times (0,T].
    \end{equation*}
    Indeed,
    \begin{equation*}
        \frac{\partial}{\partial \lambda} f(x,\lambda) = -\frac{1}{\epsilon} f \frac{\partial}{\partial \lambda} u^\epsilon(x,\lambda).
    \end{equation*}
    On the other hand, recall that $\Delta_x f(x,\lambda)=\mathrm{tr}(D^2_x f(x,\lambda))$. We have
    \begin{align*}
        D^2_x f & = \nabla_x \left(-\frac{1}{\epsilon} f \nabla_x u^\epsilon \right) \\
        & = -\frac{1}{\epsilon}\left[ \nabla_x f \otimes \nabla_x u^\epsilon + f D^2_x u^\epsilon \right] \\
        & = -\frac{1}{\epsilon}\left[ -\frac{1}{\epsilon} f (\nabla_x u^\epsilon \otimes \nabla_x u^\epsilon) + f D^2_x u^\epsilon\right].
    \end{align*} 
    Hence,
    \begin{equation*}
        \Delta_x f = -\frac{1}{\epsilon} \left[-\frac{1}{\epsilon} f \Vert \nabla_x u^\epsilon \Vert^2 + f \Delta_x u^\epsilon \right].
    \end{equation*}
    Now,
    \begin{align*}
        \frac{\partial}{\partial \lambda} f - \frac{\epsilon}{2}\Delta_x f & = -\frac{1}{\epsilon} f \frac{\partial}{\partial \lambda} u^\epsilon + \frac{1}{2} \left[-\frac{1}{\epsilon} f \Vert \nabla_x u^\epsilon \Vert^2 + f \Delta_x u^\epsilon \right] \\
        & = -\frac{1}{\epsilon} f \left[ \frac{\partial}{\partial \lambda} u^\epsilon + \frac{1}{2} \Vert \nabla_x u^\epsilon \Vert^2 - \frac{\epsilon}{2} \Delta_x u^\epsilon\right] \\
        & = 0.
    \end{align*}
    Any positive, bounded solution to this equation satisfies  \cite[Theorem 1.1]{helmensdorfer2013geometry}
\begin{equation*}
    D^2_x (\log f)(x,\lambda) + \frac{1}{\epsilon \lambda} \Id \succeq  0 \quad \text{for all }\lambda>0,
\end{equation*}
hence,
\begin{equation*}
    D^2_x(u^\epsilon)(x,\lambda) \preceq  \frac{1}{\lambda}.
\end{equation*}
\end{proof}
This result highlights the Lipschitz constant of the gradient of the approximation does not depend on the value of $\epsilon$. Hence, if $u^\epsilon$ is convex, we obtain that $\g_f$ is nonexpansive.
\begin{proposition}
    Let $\phi$ be a proper, l.s.c., convex function. Set $\lambda>0$ and $\tau\leq2/\lambda$. Then $u^{\epsilon}$ is convex and,
    \begin{equation}
        (\forall x,y \in \RR^N), \quad \Vert \g_f(x) - \g_f(y) \Vert \leq \Vert x - y\Vert.
    \end{equation}
\end{proposition}
\begin{proof}
    We first show that $u^{\epsilon}$ is convex. The Cole-Hopf change of variable $f= \exp\left(-\frac{u^{\epsilon}}{\epsilon}\right)$ solves the Heat equation, and can be expressed explicitly from $\phi$ using
    \begin{equation*}
        f(x,\lambda) = G_{\epsilon \lambda} \ast \exp(-\frac{\phi}{\epsilon}),
    \end{equation*}
    where $G_{\epsilon \lambda} := \frac{1}{(2\pi \epsilon \lambda)^{N/2}} \exp\left(-\frac{\Vert \cdot \Vert^2}{2 \epsilon \lambda}\right)$.
    Log-concavity is preserved by convolution \cite[Proposition 3.5]{saumard2014log}, hence $\phi$ being convex implies $f$ being log-concave. Therefore, $u^{\epsilon}$ is convex.
    By Baillon-Haddad theorem, we have
    \begin{equation*}
        \langle \nabla u^{\epsilon}(x) -\nabla u^{\epsilon}(y), x-y \rangle \geq \lambda \Vert \nabla u^{\epsilon}(x) -\nabla u^{\epsilon}(y)\Vert^2,
    \end{equation*}
    Therefore,
    \begin{align*}
        \Vert \g_f(x) - \g_f(y) \Vert^2 & = \Vert (\Id- \tau \nabla u^{\epsilon})(x) - (\Id - \tau \nabla u^{\epsilon})(y) \Vert^2 \\
        & = \Vert x- y\Vert^2 - 2 \tau \langle \nabla u^{\epsilon}(x) - \nabla u^{\epsilon}(y), x-y \rangle + \tau^2 \Vert \nabla u^{\epsilon}(x) - \nabla u^{\epsilon}(y) \Vert^2 \\
        & \leq \Vert x- y\Vert^2 -2\tau \lambda \Vert \nabla u^{\epsilon}(x) - \nabla u^{\epsilon}(y) \Vert^2 + \tau^2 \Vert \nabla u^{\epsilon}(x) - \nabla u^{\epsilon}(y) \Vert^2 \\
        & \leq \left(1-\frac{2\tau}{\lambda} + \frac{\tau^2}{\lambda^2}\right) \Vert x-y \Vert^2 \\
        & = \left(1-\frac{\tau}{\lambda}\right)^2 \Vert x- y \Vert^2\\
        & \leq \Vert x- y \Vert^2.
    \end{align*}
\end{proof}
This proposition cannot be readily extended to weakly convex functions. Indeed, weak convexity is not preserved by convolution. Nonetheless, when $\phi$ is not convex, we still have Lipschitz continuity of $\g_f$.
\begin{proposition}
    Suppose that Assumption \ref{ass:weaklyconvex} holds. Set $\lambda>0$ and $\tau<1/\lambda$. Then,
    \begin{equation}
        (\forall x,y \in \Hi), \quad \Vert \g_f(x) - \g_f(y) \Vert \leq \left(1+\frac{\tau}{\lambda}\right) \Vert x - y \Vert.
    \end{equation}
\end{proposition}
\begin{proof}
    We have 
    \begin{align*}
        \Vert \g_f(x) - \g_f(y) \Vert &= \Vert x- \tau\nabla u^{\epsilon}(x)-y + \tau\nabla u^{\epsilon}(y) \Vert \\
        & \leq \Vert x- y \Vert + \tau \Vert \nabla u^{\epsilon}(x)- \nabla u^{\epsilon}(y) \Vert \\
        & \leq \left(1+\frac{\tau}{\lambda}\right)\Vert x-y \Vert.
    \end{align*}\end{proof}

\paragraph{Admissibility of $\g_f$.} With the solution to the Heat equation obtained with Cole-Hopf change of variables, we recover the solution $u^\epsilon$ to \eqref{eq:HJ_viscous}:
\begin{equation}
   (\forall x,\lambda \in \Hi \times [0,\Lambda]), \quad u^\epsilon(x,\lambda) = -\epsilon \log \left( G_{\epsilon \lambda} \ast \exp\left(-\frac{\phi}{\epsilon} \right)(x)\right)  
\end{equation} 
If $\phi$ is convex, then by preservation of log-concavity by the convolution, $u^\epsilon$ is convex w.r.t. $x$. Moreover, if $\phi$ is inf-compact, i.e., the level sets $[\phi \leq \alpha]$ are compact for all $\alpha \in \RR$ \cite{VarAnalRockafellar}, $u^\epsilon$ is also inf-compact. We prove this statement by showing the equivalent statement $u^\epsilon(x) \to +\infty$ when $\Vert x \Vert \to +\infty$.
\begin{proposition} \label{prop:admissibility_f}
    Let $\phi$ be a proper, l.s.c., convex, inf-compact function. Then the solution $u^\epsilon(\cdot, \lambda)$ of \eqref{eq:HJ_viscous} is proper, l.s.c., convex and inf-compact for all $\lambda \in [0,T]$.
\end{proposition}
\begin{proof}
    The result is trivial for $\lambda = 0$ as $u^\epsilon(\cdot,0)=\phi$. Fix $\lambda\in(0,T]$. We have 
    \begin{align*}
   (\forall x \in \Hi), \quad u^\epsilon(x,\lambda) = -\epsilon \log \left(\frac{1}{(2\pi\epsilon\lambda)^{N/2}}\int_{\RR^N} \exp\left(-\frac{1}{\epsilon}\phi(y) - \frac{1}{2\epsilon\lambda}\Vert x-y \Vert^2 \right) dy\right).
\end{align*} 
By Fatou's lemma, $u^\epsilon(\cdot, \lambda)$ is l.s.c.. as the integrand is continuous in both $x$ and $y$, hence measurable and l.s.c.. Furthermore, $\phi$ is lower bounded, hence the integral in the log is always finite. As it is also positive, $u^\epsilon(\cdot,\lambda)$ is proper. For the inf-compactness, we can show that
\begin{align*}
  \lim_{\Vert x \Vert \to +\infty} I(x) := \int_{\RR^N} \exp\left(-\frac{1}{\epsilon}\phi(y) - \frac{1}{2\epsilon\lambda}\Vert x-y \Vert^2 \right) dy = 0,
\end{align*} 
to conclude. Let $R>0$, and $x$ such that $\Vert x \Vert \geq R$. By inf-compactness of $\phi$, there exists $M \in \RR$ such that $[\phi \leq M] \subseteq \mathbb{B}(0,R)$. We split $I$ into $I(x) = I_1(x) + I_2(x)$ where 
\begin{equation*}
    I_1(x) = \int_{[\phi \leq M]} \exp\left(-\frac{1}{\epsilon}\phi(y) - \frac{1}{2\epsilon\lambda}\Vert x-y \Vert^2 \right) dy,
\end{equation*}
and 
\begin{equation*}
    I_2(x) = \int_{[\phi > M]} \exp\left(-\frac{1}{\epsilon}\phi(y) - \frac{1}{2\epsilon\lambda}\Vert x-y \Vert^2 \right) dy.
\end{equation*}
Immediately, we have $I_2(x) \leq e^{\frac{-1}{\epsilon}M} (2\pi \epsilon \lambda)^{N/2}$. On the other hand as $\Vert x-y \Vert \geq \Vert x \Vert - R$ if $y \in [\phi(y)\leq M]$, we have 
\begin{equation*}
    I_1(x) \leq e^{\frac{-1}{\epsilon}\left(\inf \phi + \frac{1}{2\lambda}\left(\Vert x \Vert - R \right)^2\right) } \int_{[\phi \leq M]}dy. 
\end{equation*}
Therefore, we obtain
\begin{equation*}
    I(x) \leq e^{\frac{-1}{\epsilon}M} (2\pi \epsilon \lambda)^{N/2} + \lambda^N([\phi \leq M])e^{\frac{-1}{\epsilon}\left(\inf \phi + \frac{1}{2\lambda}\left(\Vert x \Vert - R \right)^2\right) }.
\end{equation*}
For any value $\delta>0$ we can set $M$ so that $I_2(x) \leq \frac{\delta}{2}$ and similarly we can choose $R$ so that $I_1(x) \leq \frac{\delta}{2}$ (again with $\Vert x \Vert \geq R$), hence choosing both $M$ and $R$ large enough yields $I(x) \leq \delta$ and thus $I(x) \rightarrow 0$ when $\Vert x \Vert \rightarrow +\infty$. 
\end{proof}
The admissibility is now readily available, but we can only guarantee that $\inf u^\epsilon \rightarrow \inf u$ as $\epsilon \rightarrow 0$, without convergence rates as $u^\epsilon$ converges uniformly to $u$ \cite{crandall1984two}.

\subsection{Admissibility of the approximations}
In this section, we discuss the admissibility of the different type of approximations. As we have seen in their respective presentation, showing this property without assumptions on $\epsilon$ is not possible in general for approximations of type (a), (b), and (d). However it is fairly reasonable to assume that these approximations can be constructed in a way that guarantees admissibility, and we support this claim by constructing approximations of the proximal operator of $\phi:=\frac{1}{2}\Vert \cdot \Vert_2^2$. Other examples can be found in Appendix \ref{app:admissible_examples}.
\paragraph{The example of the squared $\ell_2$-norm}
We set $\phi := \frac{\lambda}{2}\Vert \cdot \Vert^2_2$, $\lambda>0$. Its proximal operator is given for all $x \in \Hi$ by
\begin{equation}
    \prox_{\phi}(x)=\frac{x}{\lambda + 1}.
\end{equation}
Moreover, $\psi$ is defined as
\begin{equation}
    (\forall x \in \Hi), \quad \psi(x) = \frac{1}{2(\lambda+1)}\Vert x \Vert^2_2 + C.
\end{equation}
with some $C>0$. The $\epsilon$-subdifferential of $\phi$ is given by \cite[Example XI.1.1.2]{hiriart-urruty1993convexII}
\begin{equation*}
    (\forall x \in \Hi), \quad \partial_\epsilon \phi (x) = \lambda \{x + s \mid \Vert s \Vert^2_2 \leq \epsilon/\lambda \},
\end{equation*}
and similarly the $\epsilon$-subdifferential of $\psi$ is given by
\begin{equation}
    (\forall x \in \Hi), \quad \partial_\epsilon \psi(x) = \frac{1}{\lambda+1}\{x + s \mid \Vert s \Vert^2_2 \leq \epsilon(\lambda+1) \},
\end{equation}
\begin{itemize}
    \item Take $\g_a$ an $\epsilon$-approximation of type (a) of $\prox_\phi$ as
    \begin{equation}
        \g_a(x) = \prox_\phi(x) + e
    \end{equation}
    where $\Vert e \Vert_2 \leq \epsilon$.
    \item Take $\g_b$ an $\epsilon$- approximation of type (b) of $\prox_\phi$ as
\begin{equation}
    \g_b(x) = \prox_\phi(x + r)
\end{equation}
where $\Vert r \Vert_2 = \epsilon$.
    \item Take $\g_d$ an $\epsilon$-approximation of type (d) of $\prox_\phi$ as
    \begin{equation}
        \psi_\epsilon(x) = \psi(x) +\epsilon \exp\left(-\frac{\Vert x-e \Vert^2}{2}\right)
    \end{equation}
    with $\Vert e \Vert = \epsilon$. Hence,
    \begin{equation}
        \g_d(x) =\frac{x}{1+\gamma} - \epsilon (x-e) \exp \left(-\frac{\Vert x-e \Vert^2}{2} \right).
    \end{equation}
\end{itemize}

\begin{proposition} The approximations of $\prox_\phi$ are such that
    \begin{enumerate}[label=(\alph*)]
    \item $y \in \Fix \g_a \Leftrightarrow y= \frac{1+\gamma}{\gamma}e$. The fixed point is unique.
    \item $y \in \Fix \g_b \Leftrightarrow y = \frac{r}{\gamma}$. The fixed point is unique.
    \item There exists $t \in (-1,1)$ such that $y\in \Fix \g_d \Leftrightarrow y = te$.
    \end{enumerate}
\end{proposition}
\begin{proof}
    \begin{enumerate}[label=(\alph*)]
        \item Straightforward. 
        \item Straightforward. 
    \item Finding a fixed point of $\g_d$ is equivalent as finding $t\in \RR$ such that $y= t e$. By the intermediate value theorem, $t \in (-1,1)$. 
    \end{enumerate}
\end{proof}

\newpage

\section{Convergence of inexact proximal algorithms}
Now that we have studied the properties of the approximations of the proximal operator, we study the conditions under which these approximations can be incorporated in proximal algorithms while guaranteeing convergence. %
We will first consider the case of the proximal point algorithm, the easiest to study, before taking into account the term $f$. In this section, we will assume that the approximations are admissible (Definition \ref{def:admissible}). Or, when composed with the operator of the data fidelity term, that fixed points of the composition exist if we cannot prove that they do without specifying the approximations.

\subsection{Proximal point algorithms}
In this section, we investigate the convergence guarantees associated with each approximation when we iterate proximal point algorithms. Given that for all approximations, we showed in the previous section its $(L_\g,\gamma)$-Lipschitzness and its $\sigma(\epsilon)$ approximation power, we provide general convergence guarantees before specifying the results.

\paragraph{Strongly convex $\phi$.}
If we assume strong convexity of $\phi$, its proximal operator is a strict contraction: $L_\psi = \frac{1}{1+\mu}$ \cite{bauschke2012firmly}. Hence, we can study the sequence generated by 
\begin{equation} \label{eq:prox_point}
    x_{k+1} = \g_k(x_k),
\end{equation}
as $\g_k$ is a strict contraction up to $\gamma_k$, and approximates $\prox_\phi$ with error $\epsilon_k$. We study the vanishing case by letting $\gamma_k$ and $\epsilon_k$ decrease along the iterates with the limit condition $\lim_{k\to+\infty}\gamma_k+\sigma(\epsilon_k)=0$. 
\begin{proposition} \label{prop:proximal_point}
    Let $\phi$ be a proper, l.s.c., $\mu>0$-strongly convex function. Let $(\g_k)_{k\in\mathbb{N}}$ be a sequence of $(\epsilon_k)$-approximations of $\prox_\phi$ with $(L_\g,\gamma_k)$-Lipschitzness ($L_\g<1$). Then the sequence $(x_k)_{k\in\mathbb{N}}$ generated by \eqref{eq:prox_point}, converges to
    \begin{enumerate}
        \item a ball of radius $\frac{(\gamma+\sigma(\epsilon))}{1-L_\g}$ around the minimizer of $\phi$, if $\gamma_k\leq \gamma$ and $\epsilon_k \leq \epsilon$ for all $k$,
        \item the minimizer of $\phi$, if $\lim_{k\to+\infty}\gamma_k+\sigma(\epsilon_k)=0$. %
    \end{enumerate}
\end{proposition}
\begin{proof}
    We look at the residual with $x^* \in \Fix \prox_\phi$. 
    \begin{align*}
        \Vert x_{k+1}-x^* \Vert & = \Vert \g_k(x_k) -\prox_\phi(x^*) \Vert \\
        & = \Vert \g_k(x_k) -g_k(x^*) + g_k(x^*)-\prox_\phi(x^*) \Vert \\
        & \leq  L_\g\Vert x_k - x^* \Vert +\gamma_k + \sigma(\epsilon_k) \\
        & \leq L_\g \left( L_\g\Vert x_{k-1} - x^* \Vert +\gamma_{k-1} +\sigma(\epsilon_{k-1}) \right) +\sigma(\epsilon_k) \\
        & \leq L_\g^{k+1} \Vert x_0 - x^* \Vert + \sum_{i=0}^k \left(\gamma_i + \sigma(\epsilon_i)\right)L_\g^{k-i} \\
        &= L_\g^{k+1} \Vert x_0 - x^* \Vert + \sum_{j=0}^k \left(\gamma_{k-j} + \sigma(\epsilon_{k-j})\right)L_\g^{j}
    \end{align*} 
    If $\gamma_k = \gamma$ and $\epsilon_k = \epsilon$ for all $k$, then,
    \begin{align*}
        \Vert x_{k+1}-x^* \Vert & \leq L_\g^{k+1} \Vert x_0 - x^* \Vert + (\gamma+\sigma(\epsilon)) \frac{1-L_g^{k+1}}{1-L_\g} \\
        & \xrightarrow[k\to+\infty]{} (\gamma+\sigma(\epsilon)) \frac{1}{1-L_\g}.
    \end{align*}
    If $\gamma_k,\epsilon_k\to 0$ as $k \to +\infty$, then, 
    \begin{align*}
        \Vert x_{k+1}-x^* \Vert & \leq L_\g^{k+1} \Vert x_0 - x^* \Vert + \sum_{j=0}^k \left(\gamma_{k-j} + \sigma(\epsilon_{k-j})\right)L_\g^{j} \\
        & \xrightarrow[k\to+\infty]{} 0.
    \end{align*}
    Indeed, $\sum_{j=0}^k \left(\gamma_{k-j} + \sigma(\epsilon_{k-j})\right)L_\g^{j} \to 0$ when $k$ goes to infinity. First, as $\gamma_k+\sigma(\epsilon_k) \to 0$, $(\gamma_k+\sigma(\epsilon_k))_{k\in\mathbb{N}}$ is bounded by a quantity $C>0$. Second, $\sum_{j\geq 0} L_g^j$ is also bounded by a quantity $M>0$. Hence, for $\eta>0$, there exists $J(\eta)$ such that
    \begin{equation*}
        \sum_{j> J(\eta)} L_\g^{j}  \leq \frac{\eta}{2C},
    \end{equation*}
    thus, for all $K > J(\eta)$, we have 
    \begin{equation*}
        \sum_{j> J(\eta)}^K \left(\gamma_{K-j}+\sigma(\epsilon_{K-j})\right)L_\g^{j} \leq \frac{\eta}{2}.
    \end{equation*}
    Now by the fact that $\gamma_k+\sigma(\epsilon_k) \to 0$ as $k$ goes to infinity, there exists $K(\eta)$ such that for all $j \in \{0,1,\ldots,J(\eta)\}$ and for all $k \geq K(\eta)$
    \begin{equation*}
        \gamma_{k-j}+\sigma(\epsilon_{k-j}) \leq \frac{\eta}{2 M}.
    \end{equation*}
    Therefore, for all $k \geq K(\eta)$
    \begin{equation*}
        \sum_{j=0}^{J(\eta)} \left(\gamma_{k-j}+\sigma(\epsilon_{k-j})\right)L_\g^j \leq \frac{\eta}{2 M} M = \frac{\eta}{2}.
    \end{equation*}
    As these statements hold for any $\eta>0$ we have
    \begin{equation*}
        \sum_{j=0}^k \left(\gamma_{k-j}+\sigma(\epsilon_{k-j})\right)L_\g^j \xrightarrow[k\to+\infty]{} 0.
    \end{equation*}
\end{proof}
This result is probably the best one can hope for approximation with non-zero $\gamma$ relaxation of the Lipschitz continuity, and non summable errors. %
\paragraph{Convex $\phi$.}
In a convex setting, the approximations can be at best non-expansive, hence the previous convergence analysis does not hold. The following Krasnosel'ski\u{\i}--Mann iterations, with $x_0 \in \Hi$ and for $k=0,1,\ldots,$
\begin{align}
    y_k & = x_k + \alpha_k (x_k - x_{k-1}) + \epsilon_k \nonumber \\
    z_k &= x_k + \beta_k (x_k - x_{k-1}) + \rho_k \nonumber \\
    x_{k+1} &= (1-\lambda_k)y_k + \lambda_k \g z_k + \theta_k \label{alg:peypouquet}
\end{align}
where the hyperparameters follows Assumption \ref{ass:km}.
If $\Fix \g\neq \emptyset$, and $g:K \to K$ is quasinonexpansive, then by \cite{cortild2025krasnoselskii} (Theorem \ref{th:peypouquet}), $(x_k)_{k \in \mathbb{N}}$ converges strongly to a point in $\Fix \g$ (if they exist). Therefore, we directly have the following result
\begin{proposition}
    Let $\phi:\Hi\to\overline{\RR}$ be a proper, l.s.c, convex function. Let $\g$ be an admissible $\epsilon$ approximation of $\prox_\phi$, with regularity $(1,0)$. Then, the sequence $(x_k)_{k\in \mathbb{N}}$ generated by \eqref{alg:peypouquet}, starting from $x_0 \in \Hi$, converges to a point in $\Fix \g$.
\end{proposition}

Without any more knowledge on our approximations, only approximations of type (d) and (f) may satisfy these assumptions. The rest of the approximation have a non-zero $\gamma$ error in their Lipschitz continuity, therefore they require either more explicit construction or $\mu>0$-strong convexity of $\phi$ in order to converge to a solution. 

\subsection{Proximal splitting algorithms}
In this section we aim at minimizing the sum of $f$ and $\phi$ with proximal splitting algorithms. We can regroup the convergence analysis of the algorithms by proving either that the operator (regrouping the action of the operators of $f$ and $\phi$) is quasi non-expansive, or a contraction. In the following we will make the following assumption:
\begin{assumption}\label{ass:strong_convexity} $f$ is $\mu$-strongly convex and $L_f$-smooth.
\end{assumption}
The value of $\mu$ will allows us to counteract the higher Lipschitz constant of the approximations. We then study the following optimization operators:
\paragraph{Forward-backward splitting.}
We replace in the usual forward-backward splitting the proximal operator of $\phi$ by $\g$ so that:
\begin{equation} \label{eq:T_FB}
    T_{\tau FB} = \g\left(\Id - \tau \nabla f \right).
\end{equation}
\paragraph{Peaceman-Rachford splitting.} We replace the reflection operator of $\prox_\phi$ by $2\g-\Id$, to define symmetrically the Peaceman-Rachford splitting:
\begin{equation}
    T_{PR,\g,\tau f} = (2\g -\Id) \circ (2\prox_{\tau f} -\Id), \text{ and } T_{PR,\tau f,\g} = (2\prox_{\tau f} -\Id)\circ(2\g-\Id).
\end{equation}
In a typical convex optimization setting, the order does not matter w.r.t. the solutions, but in our case, as $\g$ is $(L_\g,\gamma)$-Lipschitz continuous, the order matters. 
\paragraph{Douglas-Rachford splitting.} The Douglas-Rachford splitting is the average between $\Id$ and the $T_{PR}$, i.e.,
\begin{equation}
    T_{DR, \g,\tau f} = \frac{\Id+T_{PR, \g, \tau f}}{2}, \text{ and }  T_{DR, \tau f ,\g} =\frac{\Id + T_{PR,\tau f, \g}}{2}.
\end{equation}
We display in Figure \ref{fig:Lipschitz_constants} the Lipschitz continuity of these operators with respect to the values of $L_\g$ and of $\mu/L_f$ if $\gamma = 0$.
\begin{figure}[ht]
  \centering
  \begin{subfigure}{0.32\textwidth}
    \includegraphics[width=\linewidth]{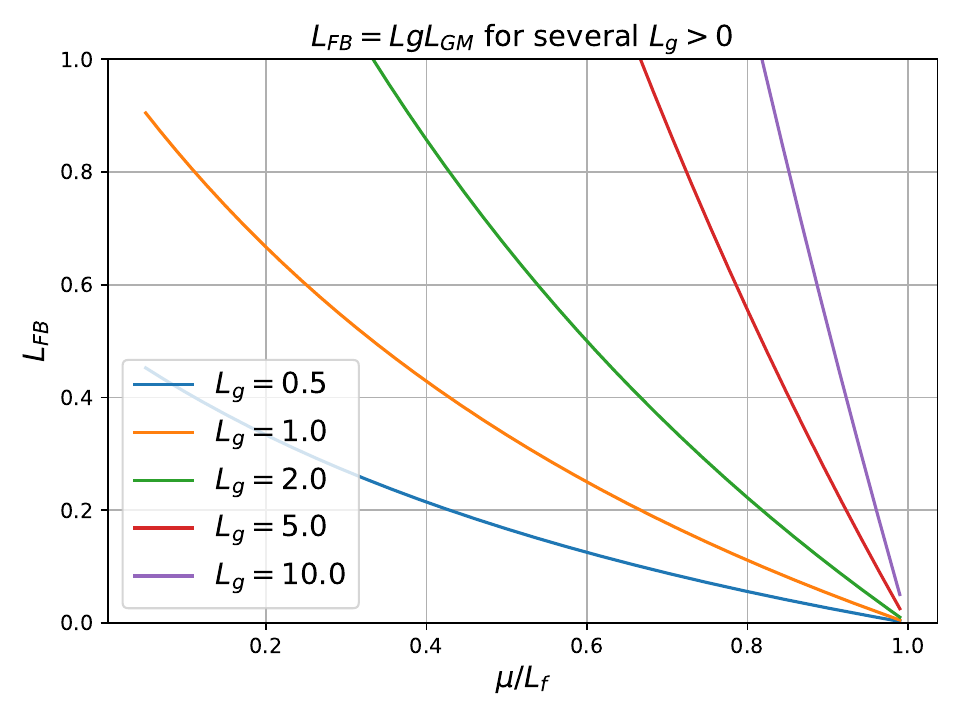}
    \caption{FB contractivity}
    \label{fig:a}
  \end{subfigure}
  \begin{subfigure}{0.32\textwidth}
    \includegraphics[width=\linewidth]{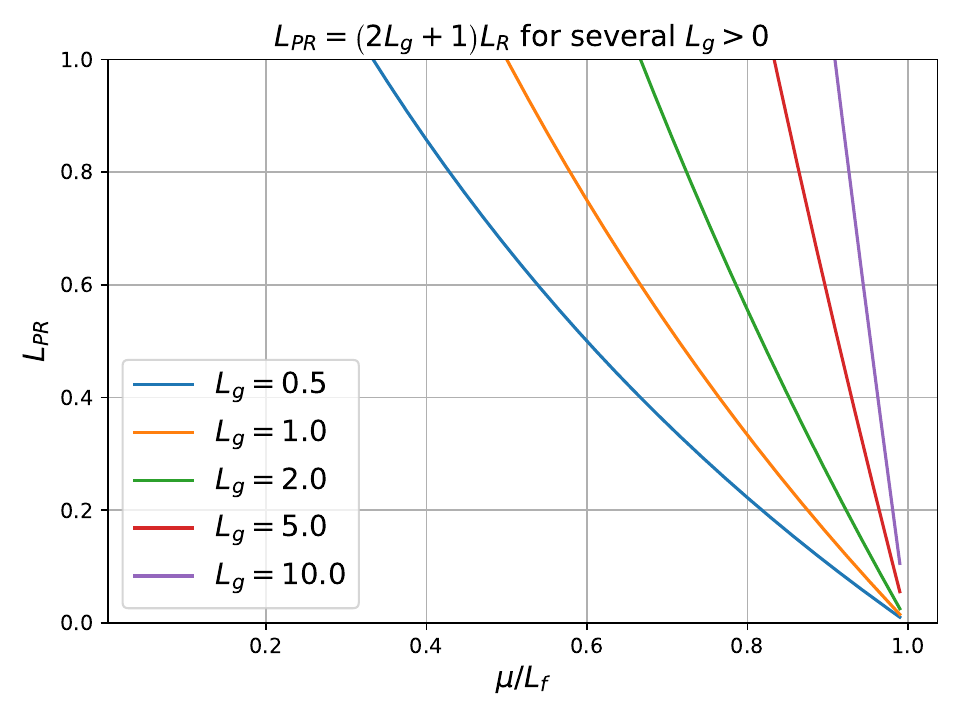}
    \caption{PR contractivity}
    \label{fig:b}
  \end{subfigure}
  \begin{subfigure}{0.32\textwidth}
    \includegraphics[width=\linewidth]{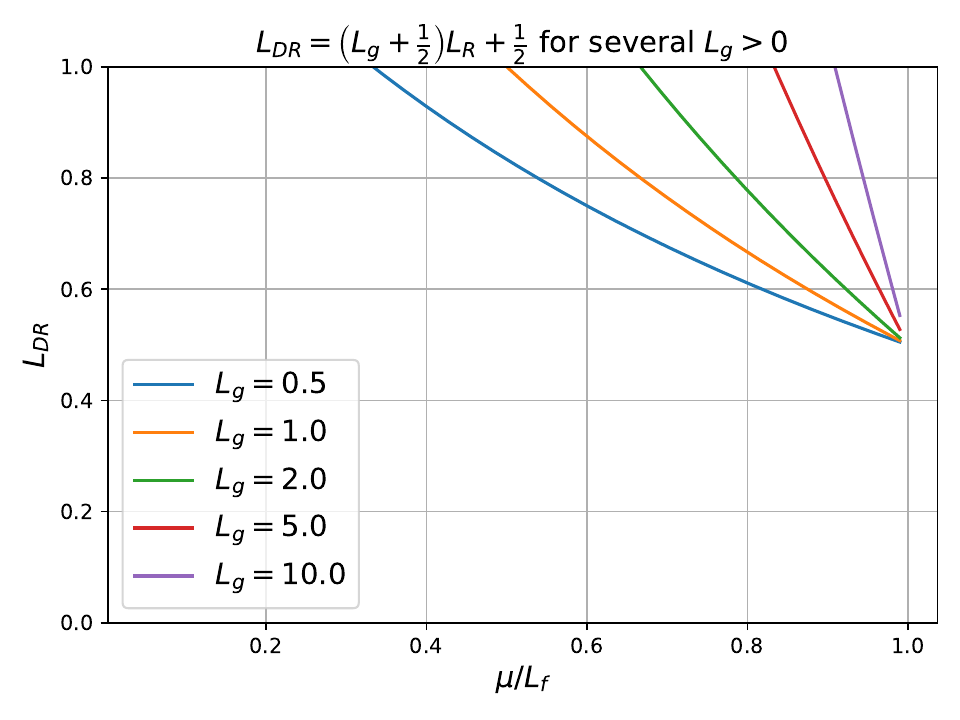}
    \caption{DR contractivity}
    \label{fig:c}
  \end{subfigure}
  \caption{Contractivity of the operators, with respect to the Lipschitz constant of $\g$ and the strong convexity of $f$, when the continuity is exact and varying $\mu/L_f$ for the optimal value of $\tau$ \cite{ryu2016primer}.}
  \label{fig:Lipschitz_constants}
\end{figure}
One can see that even in the case where $L_\g$ is strictly greater than $1$, existence and convergence to a solution can be obtained for the three algorithms. Note for instance that $L_\g=10$ corresponds to $0.9$-weakly convex function and recall that $1$-weakly convex function do not have Lipschitz continuous proximal operators. 

This is a promising result, notably for plug-and-play methods where contractivity of $\g$ is required in one way or another for convergence to happen (albeit in more general settings) \cite{hurault2022proximal,terris2021learning,bredies2024learning}. Using the computations done in the proof of Proposition \ref{prop:proximal_point}, we infer the convergence guarantees of the three splitting algorithms.

\begin{proposition}{Convergence of Forward-backward splitting.} Suppose that Assumption \ref{ass:strong_convexity} holds.  Let $(\g_k)_{k\in\mathbb{N}}$ be a sequence of $(\epsilon_k)$-approximations of $\prox_\phi$ with $(L_\g,\gamma_k)$-Lipschitzness. Then, the sequence $(x_k)_{k\in \mathbb{N}}$ generated by $x_0 \in \Hi$ and $x_{k+1} = T_{\tau FB_k}(x_k)$ with $\tau >0$ converges to
    \begin{enumerate}
        \item a ball of radius $\frac{\gamma+\sigma(\epsilon)}{1-L_{FB}}$ around the minimizer of $f+\phi/\tau$, if $\gamma(x_k)\leq \gamma$ and $\epsilon_k \leq \epsilon$ for all $k$,
        \item the minimizer of $f+\phi/\tau$, if $\lim_{k\to+\infty}\gamma(x_k)+\sigma(\epsilon_k)=0$. 
    \end{enumerate}
    if 
    \begin{equation*}
        L_{FB} := L_\g L_{GM}<1.
    \end{equation*}
\end{proposition}
\begin{proof}
    The unique minimizer of $f+\phi/\tau$, $x^*$ is such that
    \begin{align*}
        \Vert x_{k+1} - x^* \Vert = \Vert \g_k(x_k - \tau \nabla f(x_k)) - \prox_\phi(x^* - \tau \nabla f(x^*))\Vert
    \end{align*}
    The rest follows from the proof of Proposition \ref{prop:proximal_point}.
\end{proof}
\begin{proposition}{Convergence of Peaceman-Rachford splitting.} Suppose that Assumption \ref{ass:strong_convexity} holds. Let $(\g_k)_{k\in\mathbb{N}}$ be a sequence of $(\epsilon_k)$-approximations of $\prox_\phi$ with $(L_\g,\gamma_k)$-Lipschitzness. Then, the sequence $(x_k)_{k\in \mathbb{N}}$ generated by $x_0 \in \Hi$ and $x_{k+1} = T_{PR,\g_k,\tau f}(x_k)$ with $\tau >0$ converges to
    \begin{enumerate}
        \item a ball of radius $\frac{2(\gamma+\sigma(\epsilon))}{1-L_{PR,\g,f}}$ around the minimizer of $f+ \phi/\tau$, if $\gamma(x_k)\leq \gamma$ and $\epsilon_k \leq \epsilon$ for all $k$,
        \item the minimizer of $f+\phi/\tau$, if $\lim_{k\to+\infty}\gamma(x_k)+\sigma(\epsilon_k)=0$,%
    \end{enumerate}
    And, the sequence $(x_k)_{k\in \mathbb{N}}$ generated by $x_0 \in \Hi$ and $x_{k+1} = T_{PR,\tau f,\g_k}(x_k)$ with $\tau >0$ converges to
    \begin{enumerate}
        \item a ball of radius $\frac{2 L_R (\gamma+\epsilon)}{1-L_{PR,\g,f}}$ around the minimizer of $f+ \phi/\tau$, if $\gamma(x_k)\leq \gamma$ and $\epsilon_k \leq \epsilon$ for all $k$,
        \item the minimizer of $f+ \phi/\tau$, if $\lim_{k\to+\infty}\gamma(x_k)+\sigma(\epsilon_k)=0$, %
    \end{enumerate}
    if 
    \begin{equation*}
        L_{PR,\g,f} := (2L_\g+1) L_R <1.
    \end{equation*}
\end{proposition}
\begin{proof}
    Similar steps to the proof of Proposition \ref{prop:proximal_point}.
\end{proof}
\begin{proposition}{Convergence of Douglas-Rachford splitting.} Suppose that Assumption \ref{ass:strong_convexity} holds. Let $(\g_k)_{k\in\mathbb{N}}$ be a sequence of $(\epsilon_k)$-approximations of $\prox_\phi$ with $(L_\g,\gamma_k)$-Lipschitzness. Then, the sequence $(x_k)_{k\in \mathbb{N}}$ generated by $x_0 \in \Hi$ and $x_{k+1} = T_{DR,\g_k,\tau f}(x_k)$ with $\tau >0$ converges to
    \begin{enumerate}
        \item a ball of radius $\frac{\gamma+\sigma(\epsilon)}{1-L_{DR,\g,f}}$ around the minimizer of $f+ \phi/\tau$, if $\gamma(x_k)\leq \gamma$ and $\epsilon_k \leq \epsilon$ for all $k$,
        \item the minimizer of $f+ \phi/\tau$, if $\lim_{k\to+\infty}\gamma(x_k)+\sigma(\epsilon_k)=0$, %
    \end{enumerate}
    if 
    \begin{equation*}
        L_{DR,\g,f} := (L_\g+1/2) L_R +1/2<1.
    \end{equation*}
    And, the sequence $(x_k)_{k\in \mathbb{N}}$ generated by $x_0 \in \Hi$ and $x_{k+1} = T_{DR,\tau f,\g_k}(x_k)$ with $\tau>0$ converges to
    \begin{enumerate}
        \item a ball of radius $\frac{L_R (\gamma+\epsilon)}{1-L_{DR,\g,f}}$ around the minimizer of $f+ \phi/\tau$, if $\gamma(x_k)\leq \gamma$ and $\epsilon_k \leq \epsilon$ for all $k$,
        \item the minimizer of $f+ \phi/\tau$, if $\lim_{k\to+\infty}\gamma(x_k)+\sigma(\epsilon_k)=0$. %
    \end{enumerate}
\end{proposition}
\begin{proof}
    Similar steps to the proof of Proposition \ref{prop:proximal_point}.
\end{proof}
\begin{remark}
    In contrast with what is known about proximal splitting algorithms \cite{briceno2023theoretical}, the forward-backward splitting algorithm has here the better convergence rates (as it can be seen in Figure \ref{fig:Lipschitz_constants}). This is a consequence of the averaged property of the proximal operator \cite{bauschke2017}, which can be exploited in the exact setting to obtain better rates \cite{briceno2023theoretical}. To the best of our knowledge, we should expect (with similar inexactness) these results to hold in inexact settings, but it has not been shown yet.
\end{remark}
In order to obtain an approximate solution of the original problem, we can set $\tau =1$. In all cases, this result is the first to provide a relationship between the fixed point obtained by the inexact but non-summable optimization method and its exact counterpart.

Also, the convergence results presented could probably be improved by adding additional structure on the approximations. For instance, approximations of type (d) are not assumed to be derived from a convex $\psi_\epsilon$. If we add this additional constraint, then $\nabla \psi_\epsilon$ is the proximal operator of a functional $\Tilde{\phi}$. If it is nonexpansive then convex analysis tools are available, otherwise one can use the Kurdyka-Łojasiewicz \cite{attouch2009convergence,bolte2010characterizations} to show the convergence and recover the approximate solution property using the bound $\Vert \g-\prox_\phi \Vert \leq \sigma(\epsilon)$. 

\section{Conclusion}
In this article, we conducted an exhaustive analysis of the modelling of inexactness of inexact proximal operators. We established their regularity and their approximation power with respect to the true proximal operator but also with respect to minimizers of $\phi$. An interesting consequence is that with these properties we highlighted that in a convex setting, estimating the proximal operator through the approximate subdifferential (approximations of type (c)) is the correct path if the errors can be controlled reliably (as it is done in \cite{salzo2012inexact,villa2013accelerated}). Another consequence is that when learning a proximal operator, the common approach consisting in constraining approximatively the Lipschitz constant of the approximant to be below $1$ in all cases, is detrimental to the quality of the solution obtained by the subsequent algorithms (approximations of type (d)). It remains to quantify how detrimental this is.

This analysis then allowed us to derive some new convergence results for proximal splitting algorithms in the presence of non-summable errors. We managed to show that the summability of the errors is not a necessary condition in all cases to obtain convergence, and that vanishing errors may be sufficient. We also managed to prove that convergence was still possible on non-convex problems without contractivity assumptions on the approximation of the proximal operator.

Nevertheless, we did not answer all of our introducting questions. Indeed, if the errors cannot be controlled reliably during the optimization, how can we estimate the quality of our approximate solutions? As an example, in typical learning situation (approximations of type (d)), we do not have access to the ground truth, therefore we do not have access to the error $\epsilon$. In such cases, it possible to estimate \emph{a posteriori} the error we made?

From a more optimization perspective, extension of these results to inertial methods would be of great interest, as we know that these methods have higher sensitivity to accumulated errors. Also, a drawback of our analysis is that convergence is ensured through a fixed point lens, which requires significantly stronger assumptions than typical descent methods, especially in a non-convex setting \cite{bolte2007lojasiewicz,bolte2014proximal,attouch2010proximal,attouch2013convergence}. Therefore an improvement of this work could be obtained by reducing the gap between these two ways of analyzing proximal algorithms. 

Finally, generalized proximal operators (with Bregman divergence) have been characterized in \cite{gribonval2020characterization}. Considering inexactness in this setting is a difficult task \cite{yang2022bregman}, and thus an extension of our work to this setting would be of great interest.

\bibliographystyle{plain}
\bibliography{references}

\begin{appendices}
\section{Additional results}
    \subsection{Examples of admissibility} \label{app:admissible_examples}
    In order to show that Definition \ref{def:admissible} is a reasonable assumption, we construct approximations of type (a), (b), and (d) of the proximal operators of the $\ell_2-$norm.

\paragraph{The example of the $\ell_2$-norm.}
We set $\phi := \Vert \cdot \Vert_2$. Its proximal operator of parameter $\gamma>0$ is given for all $x \in \Hi$ by
\begin{equation}
    \prox_{\gamma \phi}(x)=\left(1 - \frac{\gamma}{\max(\Vert x \Vert_2, \gamma)} \right)x
\end{equation}
We have $\Fix \prox_\phi = 0$. Moreover, $\psi$ is defined as
\begin{equation}
    \psi(x) =
    \begin{cases}
        C & \Vert x \Vert_2 \leq \gamma \\
        \frac{1}{2} \Vert x \Vert^2_2 - \gamma\Vert x \Vert_2 +C & \Vert x \Vert_2 > \gamma,
    \end{cases}
\end{equation}
with some $C>0$.  The $\epsilon$-subdifferential of $\phi$ is given by
\begin{lemma}
    The $\epsilon$-subdifferential of $\Vert . \Vert_2$ is
    \begin{equation}
        (\forall x \in \Hi), \quad \partial_\epsilon \Vert \cdot \Vert_2 (x) = \{s \mid \Vert s \Vert_2 \leq 1, ~ \Vert x \Vert_2 - \langle s,x \rangle \leq \epsilon \}.
    \end{equation}        
\end{lemma}

\begin{itemize}
    \item Take $\g_a$ an $\epsilon$-approximation of type (a) of $\prox_\phi$ as
    \begin{equation}
        \g_a(x) = \prox_\phi(x) + e
    \end{equation}
    where $\Vert e \Vert_2 \leq \epsilon$.
    \item Take $\g_b$ an $\epsilon$- approximation of type (b) of $\prox_\phi$ as
\begin{equation}
    \g_b(x) = \prox_\phi(x + r)
\end{equation}
where $\Vert r \Vert_2 = \epsilon$.
    \item Take $\g_d$ an $\epsilon$-approximation of type (d) of $\prox_\phi$ as
    \begin{equation}
        \psi_\epsilon(x) = \psi(x) +\epsilon \exp\left(-\frac{\Vert x-e \Vert^2}{2}\right)
    \end{equation}
    with $\Vert e \Vert = \epsilon$. Hence,
    \begin{equation}
        \g_d(x) = \left(1 - \frac{\gamma}{\max(\Vert x \Vert_2, \gamma)} \right)x - \epsilon (x-e) \exp \left(-\frac{\Vert x-e \Vert^2}{2} \right).
    \end{equation}
\end{itemize}

\paragraph{Approximation of type (a)}
\begin{proposition}
    The fixed point of $\g_a$, the $\epsilon$-type (a) approximation of $\prox_\phi$ are
    \begin{equation}
\Fix \g_a=
\begin{cases}
\{e\}, & \epsilon <\gamma,\\
\{\alpha e | \alpha\geq 1\}, & \epsilon =\gamma,\\
\varnothing, & \epsilon >\gamma.
\end{cases}
\end{equation}
\end{proposition}
\begin{proof}
    We have that $\g_a(x)= \prox_\phi +e$, hence
    \begin{equation}
        \g_a(x)=\left(1 - \frac{\gamma}{\max(\Vert x \Vert_2, \gamma)} \right)x +e
    \end{equation}
    The fixed point equation of $\g_a$ is 
    \begin{equation}
        x = \left(1 - \frac{\gamma}{\max(\Vert x \Vert_2, \gamma)} \right)x +e,
    \end{equation}
    which is equivalent to
    \begin{equation}
        e = \frac{\gamma}{\max(\Vert x \Vert_2,\gamma)}x.
    \end{equation}
    \textbullet Case 1: $\Vert x \Vert_2 < \gamma$. It implies $x=e$. Thus, if $\Vert e \Vert=\epsilon < \gamma$, the unique fixed point of $\g$ is $e$.
    \textbullet Case 2: $\Vert x \Vert_2 = \gamma$. Hence,
    \begin{equation}
        x = \frac{\Vert x \Vert_2}{\gamma} e \implies \Vert e \Vert_2 = \gamma \implies x = \alpha e, \quad \alpha \geq 1.
    \end{equation} 
    \textbullet Case 3 $\Vert x \Vert_2 >\gamma$. Again 
    \begin{equation}
        e = \frac{\gamma}{\max(\Vert x \Vert_2,\gamma)}x \implies \Vert e \Vert = \gamma.
    \end{equation}
    Hence, it forces to fall back on case 2. There is no fixed point if $\epsilon>\gamma$.
\end{proof}

\paragraph{Approximation of type (b)}
\begin{proposition}
    The fixed point of $\g_b$, the $\epsilon$-type (b) approximation of $\prox_\phi$ are
    \begin{equation}
\Fix \g_b=
\begin{cases}
\{0\}, & \epsilon <\gamma,\\
\{\alpha r | \alpha\geq 0\}, & \epsilon =\gamma,\\
\varnothing, & \epsilon >\gamma.
\end{cases}
\end{equation}
\end{proposition}
\begin{proof}
    We have that $\g_b(x) = \prox_\phi(x+r)$, hence
    \begin{equation}
        \g_b(x) = \left( 1 - \frac{\gamma}{\max(\Vert x +r \Vert_2,\gamma )} \right) (x + r).
    \end{equation}
    If $\epsilon \leq \gamma$ then 
    \begin{equation}
        \g_b(0) = 0.
    \end{equation}
    For $\epsilon$ sufficiently small, the unique fixed point of $\prox_\phi$ is a fixed point of $\g_b$. Now let us look at other fixed points of $\g_b$. The fixed point equation of $\g_b$ is
    \begin{equation}
         x = \left( 1 - \frac{\gamma}{\max(\Vert  x +r \Vert_2,\gamma )} \right) ( x + r).
    \end{equation}
    Set
\begin{equation}
y = x + r.
\end{equation}
Then the fixed-point equation becomes
\begin{equation} 
x = \left(1 - \frac{\gamma}{\max(\Vert y\Vert ,\gamma)}\right)y,
\quad
x = y - r.
\end{equation}
\textbullet Case 1: $\Vert y\Vert \leq \gamma$.
Thus, $x=0$. For the case to hold it requires $\Vert 0+r\Vert  = \Vert r \Vert  \leq \gamma$. Hence if $\epsilon <\gamma$ the unique fixed point of $\g_b$ is
\begin{equation}
\bar x = 0.
\end{equation}
\textbullet Case 2: $\Vert y\Vert \geq \gamma$. We have
\begin{equation}
y - r = \left(1 - \frac{\gamma}{\Vert y\Vert }\right)y
\quad\implies\quad
-r = -\frac{\gamma}{\Vert y\Vert }y
\quad\implies\quad
r = \frac{\gamma}{\Vert y\Vert }y.
\end{equation}
Thus $y$ is scaling of the error component $r$ and taking norms gives $\Vert r\Vert =\gamma$.
Let $y = t~r$ with $t\geq 1$; then
\begin{equation}
x = y - r = (t-1)r,
\end{equation}
and one checks that $_b(x)=x$. Hence if $\epsilon =\gamma$, every
\begin{equation}
\bar x = \alpha r,\quad \alpha\geq 0,
\end{equation}
is a fixed point including $\bar x=0$. Now if we suppose that $\Vert r\Vert >\gamma$, case 1 cannot hold as it would force $x=0$ while $\Vert r\Vert >\gamma$ (a contradiction),
and case 2 forces $\Vert r\Vert =\gamma$, again a contradiction. Hence if $\epsilon >\gamma$ there is no fixed point. 
To conclude, the fixed points of $\g_b$ are thus
\begin{equation}
\Fix \g_b=
\begin{cases}
\{0\}, & \epsilon <\gamma,\\
\{\alpha r | \alpha\geq 0\}, & \epsilon =\gamma,\\
\varnothing, & \epsilon >\gamma.
\end{cases}
\end{equation}

\end{proof}

\paragraph{Approximation of type (d).}
\begin{proposition}
    There exists fixed point of $\g_d$ if $\epsilon \leq \gamma$.
\end{proposition}
\begin{proof}
    We have 
    \begin{equation}
        \g_d(x) = \left( 1 - \frac{\gamma}{\max(\Vert x \Vert, \gamma)} \right) x - \epsilon (x - e) \exp\left( - \frac{\Vert x- e \Vert^2}{2} \right)
    \end{equation}
    Hence the fixed points are solutions to the equation
    \begin{equation}
        x = \frac{\epsilon \exp\left( - \frac{\Vert x- e \Vert^2}{2} \right)}{\gamma / \max(\Vert x \Vert, \gamma) + \exp\left( - \frac{\Vert x- e \Vert^2}{2} \right)} e
    \end{equation}
    All solutions are of the form $x = te$ with $t \in \RR$. Assume for now that $t\geq 0$ and that $\Vert x \Vert \leq \gamma$. We have:
    \begin{equation}
        t + \epsilon (t-1) \exp\left( - \frac{(t+1)^2\epsilon^2}{2} \right) = 0
    \end{equation}
    Denote by $f(t)$ the left hand side of this equation. $f$ is continuous and $f(0) <0$, $f(1)>0$. Hence by the intermediate value theorem there exists a solution of this equation in $(0,1)$ which is a fixed point if for instance $\epsilon \leq \gamma$. No negative solution exists as for all $t<0$, $f(t) <0$. If $\Vert x \Vert > \gamma$ then the equation becomes
    \begin{equation}
        \gamma + (t-1) \exp\left( - \frac{(t+1)^2\epsilon^2}{2} \right) = 0
    \end{equation}
    Consider
\begin{equation}
f(t)= (t-1)\exp \Bigl(-\frac{\epsilon^{2}}{2}(t+1)^{2}\Bigr).
\end{equation}
We want $t$ such that $f(t)=-\gamma$
has solutions.
\begin{equation}
\lim_{t\to\pm\infty} f(t)=0,
\quad
f(1)=0.
\end{equation}
Moreover,
\begin{equation}
f(t)<0 \quad\text{for } t<1, \qquad
f(t)>0 \quad\text{for } t>1.
\end{equation}
Hence, possible solutions lie in $(-\infty,1)$. To identify them, we compute the minimum value of $f$.
Let $a=\varepsilon^{2}/2$. Then
\begin{equation}
f'(t) = e^{-a(t+1)^{2}}\bigl[1-2a(t-1)(t+1)\bigr] = e^{-a(t+1)^{2}}\bigl[1-\varepsilon^{2}(t^{2}-1)\bigr].
\end{equation}
Setting $f'(t)=0$ gives
\begin{equation}
1-\varepsilon^{2}(t^{2}-1)=0
\quad\Longleftrightarrow\quad
t^{2}=1+\frac{1}{\varepsilon^{2}}.
\end{equation}
Thus there are two critical points
\begin{equation}
t_{\pm}=\pm\sqrt{1+\frac{1}{\varepsilon^{2}}}.
\end{equation}
We can classify them with the sign of $f''$. Since
\begin{equation}
f''(t)=e^{-a(t+1)^{2}}\Bigl[-2a(t+1)\bigl(1-\varepsilon^{2}(t^{2}-1)\bigr)
-2\varepsilon^{2}t\Bigr],
\end{equation}
and at the critical points \(1-\varepsilon^{2}(t^{2}-1)=0\), we get
\begin{equation}
f''(t_{\pm})
= e^{-a(t_{\pm}+1)^{2}}\bigl[-2\varepsilon^{2}t_{\pm}\bigr].
\end{equation}
$f''(t_-)>0$, hence $t_-$ is a global minimum. While $f''(t_+)<0$, which makes $t_+$ a global maximum. We are interested in the minimum value $f_{\min}$ which is:
\begin{equation}
f_{\min}=f(t_{-})
=(t_- - 1)\exp\!\Bigl(-\tfrac{\varepsilon^{2}}{2}(t_- -1)^{2}\Bigr) <0.
\end{equation}
Therefore there exists a solution if $\gamma \leq - f_{\min}$.
\end{proof}

\subsection{Some properties of the proximal operator.}
The following elementary result can be found for convex function in \cite{combettes2011proximal}, without proof, and we could not find any so we state it below, with a weak convexity assumption. With this property, one can "transfer" the weak convexity to a gradient term, and thus approximate the proximal operator of a convex penalty.
\begin{proposition}
    Let $\phi$ be a proper, l.s.c. function. Let $\gamma>0$ and $\alpha>0$. Then,
    \begin{equation}
        (\forall z \in \Hi), \quad \prox_{\gamma (\phi + \alpha \Vert \cdot \Vert^2)}(z) = \prox_{\gamma \phi/(2\alpha\gamma+1)}\left(\frac{z}{2\alpha \gamma +1}\right).
    \end{equation}
\end{proposition}
\begin{proof}
    Let $z\in \Hi$, we have
    \begin{align*}
        \prox_{\gamma (\phi + \alpha \Vert \cdot \Vert^2)}(z) & = \argmin_{x\in \Hi} \phi(x) + \alpha \Vert x \Vert^2 + \frac{1}{2\gamma} \Vert x-z \Vert^2 \\
    & = \argmin_{x\in\Hi} \phi(x) + \frac{2\alpha\gamma+1}{\gamma}\left(\frac{1}{2}\Vert x \Vert^2 - \langle x,\frac{z}{2\alpha\gamma+1}\rangle + \frac{1}{2(2\alpha\gamma+1)}\Vert z \Vert^2 \right) \\
    & = \argmin_{x\in\Hi} \phi(x) + \frac{2\alpha\gamma+1}{\gamma}\left(\frac{1}{2}\Vert x \Vert^2 - \langle x,\frac{z}{2\alpha\gamma+1}\rangle + \frac{1}{2}\left\Vert \frac{z}{2\alpha\gamma+1} \right\Vert^2 \right) \\
    &+ \frac{2\alpha\gamma+1}{\gamma}\Vert z \Vert^2 \left(\frac{1}{2(2\alpha\gamma+1)}-\frac{1}{2(2\alpha\gamma+1)^2}\right)\\
    & = \argmin_{x\in\Hi} \phi(x) + \frac{2\alpha \gamma +1}{2\gamma}\left\Vert x- \frac{z}{2\alpha \gamma +1}\right\Vert^2.
    \end{align*}
\end{proof}

\paragraph{From weakly convex to convex proximal operator}
Recall that for $\rho$-weakly convex functions $\prox_{\gamma \phi}$ is single valued if and only if $\gamma^{-1}-\rho>0$. We can rewrite the minimization problem of $\phi$ as 
\begin{equation*}
    \argmin_{x\in \Hi} \phi(x) = \argmin_{x\in\Hi} - \frac{\rho}{2}\Vert x \Vert^2 + \phi(x) + \frac{\rho}{2}\Vert x \Vert^2
\end{equation*}
which is the sum of a $\rho$-smooth function and a convex function, which can in turn be minimized starting from $x_0 \in \Hi$ by
\begin{equation}
    x_{k+1} = \prox_{\gamma \left(\phi + \frac{\rho}{2}\Vert \cdot \Vert^2\right)} \left(x_k - \gamma \rho x_k \right),
\end{equation}
where $0<\gamma<1/\rho$. It also means that one can replace the estimation of the proximal operator of a weakly convex function by the one of a convex function, provided that the weakly convex constant is known to adjust the scaling in the approximation.

\end{appendices}
\end{document}